\documentclass[12pt]{amsart}
\usepackage{amssymb}
\usepackage{xcolor}

\newtheorem{theorem}{Theorem}[section]
\newtheorem{proposition}[theorem]{Proposition}
\newtheorem{lemma}[theorem]{Lemma}

\newtheorem{dfn}[theorem]{Definition}
\newtheorem{definition}[theorem]{Definition}
\newtheorem{rem}[theorem]{Remark}

\def\Z{\ensuremath{\mathbb Z}}
\def\C{\ensuremath{\mathbb C}}
\def\P{\ensuremath{\mathbb P}}
\def\R{\ensuremath{\mathbb R}}

\begin{document}
\baselineskip=15pt

\title[Quasi-parabolic Higgs bundles and null hyperpolygon spaces]{Quasi-parabolic Higgs bundles and null hyperpolygon spaces}

\author[L. Godinho]{Leonor Godinho}

\address{Departamento de Matem\'atica, Centro de An\'alise Matem\'atica, Geometria e Sistemas din\^a\-micos, Instituto Superior
T\'ecnico, Universidade de Lisboa, Av. Rovisco Pais, 1049-001 Lisbon, Portugal}

\email{lgodin@math.tecnico.ulisboa.pt}

\author[A. Mandini]{Alessia Mandini}

\address{Universidade Federal Fluminense, IME-GMA, Niter\'oi, RJ, Brazil}

\email{alessia\_mandini@id.uff.br}



\thanks{}
\date{}

\begin{abstract} We introduce the moduli space of quasi-parabolic $SL(2,\C)$-Higgs bundles over a compact Riemann surface $\Sigma$  and consider a natural involution, studying its  fixed point locus when $\Sigma$ is $\C \P^1$ and establishing an identification with a moduli space of null polygons in Minkowski $3$-space. 
\end{abstract}

\maketitle

\section{Introduction}

Moduli spaces of Higgs bundles over a compact Riemann surface $\Sigma$ were introduced by Hitchin in \cite{Hi1} and have innumerable applications in many areas of mathematics and mathematical physics.  In particular, Hitchin establishes  in \cite{Hi1} a correspondence between isomorphism classes of rank-$2$  $SL(2,\C)$-Higgs bundles over $\Sigma$ and the moduli space of representations of the fundamental group of $\Sigma$ in $SL(2,\C)$. Since then these spaces have been generalized in several different ways. 

Simpson in \cite{Si1,Si2} extended these spaces to higher dimensions and, in another direction, to Higgs bundles on punctured Riemann surfaces, introducing the definition of parabolic Higgs bundles. Here he established a correspondence between isomorphism classes of these bundles and representations of the fundamental group of the punctured surface with fixed holonomy around the punctures.

On the other hand, in \cite{Hi2}, Hitchin replaced $SL(2,\C)$ by any real reductive group $G$, introducing the notion of a $G$-Higgs bundle. 
In particular, he showed  that, when $G^\C$ is a complex semisimple Lie group and $G$ is the split real form of $G^\C$, the moduli space of representations of the fundamental group of $\Sigma$ in $G$ has at least one connected component homeomorphic to a Euclidean space (often referred to as Hitchin components). 
The definition of $G$-Higgs bundles was further extended in \cite{BGM} to include parabolic $G$-Higgs bundles. 

In \cite{G-P} Garc\'ia-Prada considers involutions in the moduli space of $SL(n,\C)$-Higgs bundles and studies their fixed point sets. For example, some of the connected components of the sets obtained coincide with the moduli space of  $SL(n,\C)$-bundles and others are formed by $SL(n,\R)$-Higgs bundles. The study of involutions and their fixed point sets was further pursued in \cite{BGH}.

The fixed point sets of involutions on  moduli spaces of $G$-Higgs bundles have attracted much attention recently as they are a source of branes \cite{KW}. Indeed, as explained in \cite{HT} and \cite{DP}, the moduli space of $G$-Higgs bundles for a semisimple group $G$ and the moduli space of ${}^LG$-Higgs bundles, for  the Langlands dual ${}^LG$ of $G$, are mirror manifolds in the sense of  Strominger,  Yau and Zaslow \cite{SYZ} and the fixed point sets of involutions on these spaces are branes, i.e. special subvarieties  which can be of one of two types: $A$-branes (which are Lagrangian subvarieties) and $B$-branes  (which are complex subvarieties).
Since the moduli space of $G$-Higgs bundles is hyperk\"{a}hler it admits three complex structures\footnote{It admits a complex structure $I$ coming from the complex structure on the Riemann surface $\Sigma$, a complex structure $J$ coming from the one on the group $G^\C$ and a third one $K=I J$. The three satisfy the quaternionic equations and are compatible with a common metric.}, so it is possible to have branes that are simultaneously of type $A$ and $B$ with respect to different complex structures.

The study of involutions on the moduli space of parabolic $SL(2,\C)$-Higgs bundles on $\C P^1$ was introduced in \cite{BFGM}, where the natural  involution
\begin{equation}\label{eq:inv_0}
(E,\Phi)\to (E,-\Phi)
\end{equation}
was considered. It is holomorphic with respect to the complex structure $I$ arising from the complex structure on $\C \P^1$ and anti-holomorphic with respect to $J$ and $K$.   Using the isomorphism constructed in \cite{GM} between the moduli space of hyperpolygons (a hyperk\"{a}hler quotient of $T^* \C^{2n}$) and the moduli space of parabolic $SL(2,\C)$-Higgs bundles over $\C \P^1$ with trivial underlying vector bundle $E$, the authors study the fixed point set of \eqref{eq:inv_0} (formed by $(B,A,A)$-branes). In particular, it is shown that the non-compact  components of this manifold  correspond to $SL(2,\R)$-representations of the fundamental group of the punctured sphere (cf. \cite{BGM}) and can be identified with moduli spaces of polygons in Minkowski $3$-space with edges of fixed Minkowski lengths, with some of the edges lying in future pseudospheres and the others in past pseudospheres. This provided a nice geometrical interpretation of these $(B,A,A)$-branes in terms of moduli spaces of another related problem. 

Mirror symmetry is often associated to wall-crossing behavior. As the Langland dual of  $SL(2,\C)$ is $PGL(2,C)=SL(2,\C)/\mathbb{Z}_2$, the moduli space $H(\alpha)$ of  $SL(2,\C)$-parabolic Higgs bundles over $\C \P^1$ with parabolic weight vector $\alpha$ can be identified with its mirror.\footnote{In general, the mirror $\hat{H}(\alpha)$ of the moduli space $H(\alpha)$ of parabolic $SL(2,\C)$-Higgs bundles over a Riemann surface of genus $g$ is  $H(\alpha)/\mathbb{Z}_2^{2g}$.} In this setting, one can think of the mirror of  $\hat{H}(\alpha)$ as the moduli space $\hat{H}(-\alpha)$ of  $SL(2,\C)$-Higgs bundles over $\C \P^1$ where the parabolic weights are interchanged and the stability condition is changed accordingly.  When passing from one space to its mirror, one should expect  a wall crossing with  an intermediate step where one considers a new moduli space of $G$-Higgs bundles: the \emph{moduli space of quasi-parabolic $SL(2,\C)$-Higgs bundles over $\C \P^1$}. Quasi-parabolic  bundles were introduced by Mukai in \cite{M} and by Mehta and Seshadri in \cite{MS}. Here we generalize their definition to quasi-parabolic Higgs bundles with an appropriate definition of stability. One should think of these spaces as  limit spaces of parabolic Higgs bundles where  the two parabolic weights at each parabolic point are allowed to coincide.


In this space we consider the holomorphic involution
\begin{equation}\label{eq:inv0}
(E, \Phi) \,\longmapsto\, (E^*,\Phi^t)
\end{equation}
and give a detailed description of its fixed point set when restricted to the (generic) case when the underlying rank-$2$ vector bundle over $\C \P^1$ is trivial. We obtain $2^{n-1}-(n+1)$ connected components (where $n$ is the number of quasi-parabolic points) 
composed of quasi-parabolic $SL(2,\C)$-Higgs bundles that admit a direct sum decomposition $E=L_0\oplus L_1$ where $L_0$ and $L_1$ are trivial bundles over $\C \P^1$ and the $1$-dimensional flag components of $E$ coincide with the fibers of $L_0$ over some of the quasi-parabolic points and with the fibers of $L_1$ over the others.

To obtain this characterization of the involution fixed point set we establish a correspondence between classes of quasi-parabolic stable Higgs bundles over $\C \P^1$ and regular points of  a special singular hyperk\"{a}hler quotient of $T^* \C^{2n}$ by
$$
K \,:= \, \Big({\rm SU}(2) \times 
{\rm U}(1)^n\Big)/ (\Z/2\Z),
$$
here called \emph{null hyperpolygons}.

We further establish a correspondence between the components of the fixed point set of \eqref{eq:inv0} and moduli spaces of null polygons in Minkowski $3$-space, consisting of $SU(1,1)$-equivalence classes of closed polygons  in $\R^{2,1}$ (i.e. $\R^3$ equipped with the Minkowski inner 
product 
$$
v \circ w \,=\, -x_1x_2-y_1y_2+t_1t_2\, ,
$$
for $v=(x_1,y_1,t_1)$ and $w\,=\,(x_2,y_2,t_2)$) with some of the edges in the future light cone and the others in the past. This gives a geometrical interpretation of the $(B,A,A)$-branes obtained in terms of moduli spaces in a related space.

The paper is organized as follows. Null hyperpolygons are defined in Section~\ref{Nullhyperpolygons} as the set of regular points of a singular hyperk\"{a}hler quotient of $T^* \C^{2n}$ by the group $K$ and as a GIT quotient by $K^{\C}$. In Section~\ref{QPHBs} we define quasi-parabolic Higgs bundles and establish an isomorphism between the moduli space of null hyperpolygons and the moduli space $\mathcal{H}^n_0$ of quasi-parabolic $SL(2,\C)$-Higgs bundles over $\C \P^1$ at a divisor $D$ consisting of  $n$ marked points, where the underlying vector bundle is holomorphically trivial. In Section~\ref{sec:inv} we consider the involution in \eqref{eq:inv0} and study its fixed-point set, showing that it is formed by quasi-parabolic $SL(2,\R)$-Higgs bundles. For that we use the isomorphism of Section~\ref{QPHBs}. We then define the spaces of null polygons in Minkowski $3$-space in Section~\ref{sec:NullMink} and, in Section~\ref{sec:Mink}, we show that the connected components of the fixed point set of the involution \eqref{eq:inv0} can be identified with null polygon  spaces in $\R^{2,1}$. Finally, in Section~\ref{sec:ex}, we study the example of quasi-parabolic $SL(2,\C)$-Higgs bundles over  $\C \P^1$ with four quasi-parabolic points.

\noindent{\bf Acknowledgements}
The authors would like to thank the referee for many valuable suggestions and remarks. Both authors were partially supported by FCT/Portugal through project 
PTDC/MAT-PUR/29447/2017.

\medskip
\noindent

\section{Null Hyperpolygon spaces} \label{Nullhyperpolygons}

In this section we extend the definition of hyperpolygon spaces introduced in \cite{K2}.  


As usual, let $n$ be a positive integer and let us consider a star-shaped quiver $\mathcal Q$ with vertex and arrow  sets parametrized respectively by $I\cup 
\{0\} $ and $I$, where
$$I:=\, \{1,\ldots,n\}.$$ 
Moreover, assume that,  for each $i\, \in\, I$, the corresponding 
arrow goes from $i$ to $0$. The representations of $\mathcal 
Q$ with $V_i \,=\, \C $ for $i \,\in\, I$ and $V_0\,=\, 
\C^2$ are parametrized by
$$
E(\mathcal Q, V)\,:=\, \bigoplus_{i\in I}
\textrm{Hom} (V_i, V_0) \,=\, \C^{2n}\, .
$$
Using the standard diagonal action of ${\rm U}(2)\times{\rm U}(1)^n$ on
$E(\mathcal Q, V)$, one obtains  a  hyper\--Hamiltonian action of ${\rm U}(2)\times{\rm 
U}(1)^n$ on the cotangent bundle $T^*E(\mathcal Q, V)\,=\, T^*\C^{2n}$ for the natural hyperk\"ahler structure on $T^*E(\mathcal Q, V)$ \cite{K2}. Note that $T^*E(\mathcal Q, V)$  can be  identified with the space of representations of the doubled quiver (the quiver with the same set of vertices, where we add an arrow in the opposite direction for every arrow of $\mathcal{Q}$).

The hyper\--Hamiltonian action of ${\rm U}(2)\times{\rm 
U}(1)^n$ is not effective since every point in $T^*E(\mathcal Q, V)$ is fixed by the diagonal circle
$$
\{(\lambda\cdot \text{Id}_{\C^{2}}, \lambda,\ldots ,\lambda):\, 
\lvert \lambda\rvert \,=\, 1\}\,\subset\, {\rm U}(2) \times {\rm U}(1)^n.
$$
Therefore one considers the 
quotient group
$$K \,:=\, \Big( {\rm U}(2) \times {\rm U}(1)^n\Big)/{\rm U}
(1)\,=\, \Big({\rm SU}(2) \times 
{\rm U}(1)^n\Big)/(\Z /2\Z)\, ,$$
where $\Z /2\Z$ acts by  multiplication of each factor  by $-1$.

Let us  consider coordinates
 $(p,q)$ on $T^* \C^{2n}$, where $p=(p_1, \ldots, p_n)$ 
is the $n$-tuple of row vectors 
$$p_i =\left( \begin{array}{ll} a_i &
b_i\end{array}\right) \in (\C^2)^*$$ 
and 
$q= (q_1, \ldots, q_n)$ is the $n$-tuple of column vectors 
$q_i =\Big( \begin{array}{c}
c_i \\ d_i 
\end{array} \Big) \in \C^2$. Then,
 the  action of $K$ on $ T^*\C^{2n}$ is given by 
$$ (p,q) \cdot [A; e_1, \ldots, e_n]= 
\Big( (e_1^{-1}p_1 A, \ldots, e_n^{-1}p_n A), ( A^{-1} q_1 e_1, \ldots, A^{-1} 
q_n e_n ) \Big).$$
It  is hyper-Hamiltonian\footnote{For the symplectic forms $\omega_I$, $\omega_J$, $\omega_K$ associated to the standard triple of complex structures $I,J,K$ on  $T^*\C^{2n}$ that satisfy the quaternionic relations.} with hyperk\"ahler moment map
$$ \mu_{HK}:= \mu_{\R} \oplus \mu_{\C} : T^* \C^{2n} \longrightarrow 
\big(\mathfrak{su}(2)^* \oplus \R^n\big) \oplus \big(\mathfrak{sl}(2, 
\C)^* \oplus (\C^n)^*\big),\, $$
where $\mu_{\R}$, the real moment map,  is  
\begin{equation} \label{real} 
\mu_{\R} (p,q)\,=\,\frac{1}{2} \sum_{i=1}^n (q_i q_i^* -p_i^* p_i )_0 
\oplus \Big(-\frac{1}{2} (|q_1|^2 -|p_1|^2), \ldots, 
-\frac{1}{2} (|q_n|^2 -|p_n|^2) \Big)\, ,
\end{equation} 
with $(\,)_0$ representing the traceless part, and $\mu_{\C}$, the complex moment map,  is given by
\begin{equation} \label{complex}
\mu_{\C} (p,q)\,=\,-\sqrt{-1}  \sum_{i=1}^n (q_i p_i)_0 \oplus (\sqrt{-1}p_1 q_1, 
\ldots,\sqrt{-1} p_n q_n)\, .
\end{equation}

Let us consider the set
\begin{equation}\label{eq:P0n}
\mathcal{P}_0^n:=\left\{(p,q)\in \mu_{HK}^{-1}\left((0,0),(0,0)\right):\, \lvert p_i\rvert^2 + \lvert q_i\rvert^2 \neq 0, \,\, \forall i=1,\ldots,n \right\}.
\end{equation}

Note that an element $(p,q)\,\in\,T^*\C^{2n}$ is in $\mu_{\C}^{-1} (0,0)$ if 
and only if
\begin{equation}\label{eq:complex0}
p_i \, q_i\,=\,0 \quad \text{and} \quad \sum_{i=1}^n (q_i p_i)_0\, =\, 0,
\end{equation}
that is,  if and only if
\begin{equation}\label{complex1}
 a_i c_i + b_i d_i = 0
\end{equation}
and 
\begin{equation}\label{complex2}
\sum_{i=1}^n a_i c_i =0, \quad
\sum_{i=1}^n a_i d_i =0, \quad
\sum_{i=1}^n b_i c_i =0\, .
\end{equation}

Similarly, $(p,q)$ is in $\mu_{\R}^{-1} (0, 0) $ if and only if 
$$
\lvert q_i\rvert  = \lvert p_i\rvert \quad \text{and} \quad 
\sum_{i=1}^n \big(q_i q_i^* -p_i^* p_i \big)_0 \,=\,0\, ,
$$
i.e., if and only if
\begin{equation}\label{real1}
|c_i|^2 +|d_i|^2 =  |a_i|^2  + |b_i|^2 
\end{equation}
and 
\begin{equation}\label{real2}
\sum_{i=1}^n |c_i|^2 = \sum_{i=1}^n  |a_i|^2, \quad \sum_{i=1}^n |b_i|^2 =\sum_{i=1}^n  |d_i|^2,  \quad
\sum_{i=1}^n a_i \, \bar{b_i} - \bar{c_i}\, d_i \,=\,0\, .
\end{equation}

\
\begin{proposition}~\label{prop:p0}
The group $K$ acts freely on $\mathcal{P}_0^n$.
\end{proposition}

\begin{proof} If
$$(p,q) \cdot [A; e_1, \ldots, e_n]= (p,q)$$
for some $[A; e_1, \ldots, e_n]\in K$ and $ (p,q)\in T^* \C^{2n}$, then
$$
e_i^{-1} \, p_i \,A = p_i \quad \text{and} \quad A^{-1} \, q_i \, e_i = q_i 
$$
for $i=1,\ldots, n$, and so
$$
A  \, p_i^*= e_i \, p_i^* \, \quad \text{and} \quad A\, q_i = e_i \,q_i
$$
for $i=1,\ldots, n$. 

Since $q_i \neq 0$ and $p_i\neq 0$ (as $\lvert p_i \rvert = \lvert q_i\rvert$ and  $\lvert p_i \rvert^2+\lvert q_i\rvert^2\neq 0$ on $\mathcal{P}_0^n$), we have that $q_i$ and $p_i^*$ are eigenvectors of $A$ with eigenvalue $e_i$. If $A$ has two different eigenvalues $\lambda$ and $\lambda^{-1}$, we may assume that 
$$
A=\left[\begin{array}{cc} \lambda & 0\\ 0 &  \lambda^{-1}  \end{array}\right]
$$
and so, if $S:=\{i\in I: \, e_i= \lambda\}$, we have that
$$
q_i=\left( \begin{array}{l} c_i \\ 0\end{array} \right), \quad p_i=\left( \begin{array}{ll} a_i & 0 \end{array} \right), \quad \text{for} \quad i\in S,
$$
and
$$
q_i=\left( \begin{array}{l} 0 \\ d_i\end{array} \right), \quad p_i=\left( \begin{array}{ll} 0 & b_i  \end{array} \right), \quad \text{for} \quad i\in S^c,
$$
for some $a_i,b_i,c_i,d_i\in \C\setminus\{0\}$.

By moment map conditions \eqref{complex1} and  \eqref{real1}, we conclude that $p_i=q_i=0$ for every $i\in I$, which is impossible in $\mathcal{P}_0^n$.

We conclude that $A\in SU(2)$ has only one eigenvalue and so, since $\det A=1$, we have $A=\pm \text{Id}$. Moreover, since, for every $i\in I$, we have that $e_i$ is an eigenvalue of $A$ (as $p_i,q_i\neq 0$ for every $i$), we conclude that 
$$
[A;e_1,\ldots, e_n]= [\text{Id};1,\ldots,1].
$$ 
\end{proof}

\begin{rem}
The points in $\mu_{HK}^{-1}\left((0,0),(0,0)\right)$ for which $p_i=q_i=0$ for $i$ in some subset $S\subsetneq \{1,\ldots, n\}$ are fixed by a subtorus of $K$ of dimension $\lvert S \rvert$. If $(p,q)=(0,0)$ then this point is fixed by $K$. 
\end{rem}

Since $K$ acts freely on  $\mathcal{P}_0^n$, we have that $0$ is a regular value of the restriction of $\mu_{HK}$ to $\mathcal{P}_0^n$ and so  $\mathcal{P}_0^n$ is a smooth manifold of dimension $5(n -3)$. Moreover, by Proposition~\ref{prop:p0}, we obtain the following result. 
\begin{proposition} For $n\geq 3$ the hyperk\"{a}hler quotient
\begin{equation}\label{def:nullhp}
X_0^n:=\mathcal{P}_0^n / K
\end{equation}
is a non-empty smooth  hyperk\"{a}hler manifold of dimension $4(n-3)$.
\end{proposition}

We call $X_0^n$ the space of {\bf null hyperpolygons}.

This  space can be described as an algebro-geometric quotient by the complexified group
$$K^{\C}\,:=\, ({\rm SL}(2,\C) \times (\C^*)^n) /(\Z/2\Z),$$
of $K$. For that, we first need to give a suitable definition of stability.

%

Given $(p,q)\in T^* \C^{2n}$, a  subset $S\,\subset\,\{1, \ldots, n\}$ is called 
\emph{straight} at $(p,q)$ if $q_i$ is proportional to $q_j$ for all $i, j \,\in\, S$. 

\begin{dfn}\label{stability} A point $(p,q)\,\in \,T^* \C^{2n}$ is called
\emph{stable} 
if the following two conditions hold:
\begin{itemize}
\item[\emph{(i)}] $q_i, p_i \,\neq\, 0$ for all $i\in \{1,\ldots, n\}$ and
\item[\emph{(ii)}] the set $\{1,\ldots, n \}$ is not straight at $(p,q)$.
\end{itemize}
\end{dfn}


Let $\mu_{\C}^{-1} (0,0)^{\text{st}}$ denote the set of points in 
$\mu_{\C}^{-1} (0,0)$ which are stable. Then we have the following result.

\begin{proposition}
The group $K^{\C}$ acts freely on $ \mu_{\C}^{-1} (0,0)^{\text{\rm st}}$.
\end{proposition}

\begin{proof}  If 
$$(p,q) \cdot [A; e_1, \ldots, e_n]= (p,q)$$
for some $[A; e_1, \ldots, e_n]\in K^{\C}$ and $ (p,q)\in  \mu_{\C}^{-1} (0,0)^{\text{st}}$, then
$$
A ^T \, p_i^t= e_i \, p_i^t \, \quad \text{and} \quad A\, q_i = e_i \,q_i
$$
for $i=1,\ldots, n$.  Since, by  stability condition $(i)$, we have $q_i,p_i  \neq 0$ for all $i$,  we conclude that $q_i$ is an eigenvector of $A$ with eigenvalue $e_i$ and   $p_i^t$ is an  eigenvector of $A^T$ also with eigenvalue $e_i$.  If $A$ has two different eigenvalues $\lambda$ and $\lambda^{-1}$, we may assume that 
$$
A=\left[\begin{array}{cc} \lambda & 0\\ 0 &  \lambda^{-1}  \end{array}\right]
$$
and so, if $S:=\{i\in I: \, e_i= \lambda\}$, we have that
$$
q_i=\left( \begin{array}{l} c_i \\ 0\end{array} \right), \quad p_i=\left( \begin{array}{ll} a_i & 0 \end{array} \right), \quad \text{for} \quad i\in S,
$$
and
$$
q_i=\left( \begin{array}{l} 0 \\ d_i\end{array} \right), \quad p_i=\left( \begin{array}{ll} 0 & b_i  \end{array} \right), \quad \text{for} \quad i\in S^c,
$$
for some $a_i,b_i,c_i,d_i\in \C\setminus\{0\}$. By moment map conditions \eqref{complex1} and  \eqref{real1}, we conclude that $p_i=q_i=0$ for every $i\in I$, which is impossible by stability condition $(i)$ in Definition~\ref{stability}.


If $A$ is not diagonalizable then the eigenspace $U$ of its unique eigenvalue $\lambda$ has dimension $1$ and $q_i\in U$ for all $U$, implying that $\{1, \ldots, n\}$ is straight at $(p,q)$, which is impossible by  stability condition $(ii)$. We conclude that $A$ is diagonalizable with only one eigenvalue and then, since $\det A=1$, we have
$$
[A;e_1,\ldots, e_n]= [\text{Id};1,\ldots,1] \in K^{\C}.
$$ 
\end{proof}

Before describing $X_0^n$ as a GIT quotient we need the following definition.

\begin{definition} For each $(p,q)\in T^* \C^{2n}$, let  $d_{(p,q)}: K^{\C} \to \R$  be the function given by
\begin{equation}\label{def:function}
d_{(p,q)}([A;e_1,\ldots, e_n])= \frac{1}{4} \sum_{i=1}^n \left( \lvert \hat{q}_i \rvert^2 - \lvert q_i \rvert^2 + \lvert \hat{p}_i \rvert^2 - \lvert p_i \rvert^2 \right),
\end{equation}
where $ \hat{q}_i := A^{-1} q_i\, e_i$ and $\hat{p}_i = e_i^{-1} p_i \,A$.
\end{definition}

These functions satisfy a useful property.

\begin{lemma}\label{lemma:1}
For each $(p,q)\in T^* \C^{2n}$ let  $d_{(p,q)}: K^{\C} \to \R$  be the function in \eqref{def:function}. Then 

$d_{(p,q)}( [A;e_1,\ldots, e_n] \cdot  [B;\tilde{e}_1,\ldots, \tilde{e}_n] )$ 
\begin{align*}
 =  d_{(p,q)\cdot [A;\, e_1,\ldots, e_n]}(  [B;\tilde{e}_1,\ldots, \tilde{e}_n])+ d_{(p,q)}( [A;e_1,\ldots, e_n] )
\end{align*}
for any $ [A;e_1,\ldots, e_n],  [B;\tilde{e}_1,\ldots, \tilde{e}_n] \in K^{\C}$.
\end{lemma}

\begin{proof}
$$
d_{(p,q)}( [A; \, e_1,\ldots, e_n] \cdot  [B;\,\tilde{e}_1,\ldots, \tilde{e}_n] ) =  \frac{1}{4} \sum_{i=1}^n \left( \lvert \hat{q}_i \rvert^2 - \lvert q_i \rvert^2 + \lvert \hat{p}_i \rvert^2 - \lvert p_i \rvert^2 \right),
$$
where $ \hat{q}_i := (A B)^{-1} q_i e_i \tilde{e}_i = B^{-1} (A^{-1} q_i e_i) \tilde{e_i}$ and $\hat{p}_i = \tilde{e}_i^{-1}(e_i^{-1} p_i A) B$. Then,
\begin{align*}
& d_{(p,q)}( [A;e_1,\ldots, e_n] \cdot  [B;\tilde{e}_1,\ldots, \tilde{e}_n] ) \\ &  =  
\frac{1}{4} \sum_{i=1}^n \left( \lvert  B^{-1} (A^{-1} q_i e_i) \tilde{e_i}\rvert^2 - \lvert A^{-1} q_i e_i  \rvert^2 + \lvert \tilde{e}_i^{-1}(e_i^{-1} p_i A) B \rvert^2 -  \lvert e_i^{-1} p_i A\rvert^2 \right) \\ &
+ \frac{1}{4} \sum_{i=1}^n   \left( \lvert A^{-1} q_i e_i  \rvert^2- \lvert q_i \rvert^2 +\lvert e_i^{-1} p_i A\rvert^2- \lvert p_i \rvert^2 \right) \\ &
= d_{(p,q)\cdot [A;\, e_1,\ldots, e_n]}(  [B;\, \tilde{e}_1,\ldots, \tilde{e}_n])+ d_{(p,q)}( [A; \, e_1,\ldots, e_n] ). 
\end{align*}
\end{proof}

By the polar decomposition (a particular case of the Cartan decomposition), we can write every matrix in $SL(2,\C)$ as 
$$
A = e^{\sqrt{-1} S} R
$$
with $R\in SU(2)$  and $S\in \frak{su}(2)$ (see, for example, \cite[Theorem 6.1]{Sa}), and every element  $[A; e_1,\ldots, e_n]\in  K^{\C} \simeq  \text{exp} (\sqrt{-1} \frak{k})  \cdot K$ as
\begin{equation}\label{polardec}
[A;e_1,\ldots, e_n]=  [e^{\sqrt{-1} S};  e^{\sqrt{-1} \hat{\xi}_1},\ldots,   e^{\sqrt{-1} \hat{\xi}_n}] \cdot [R; \hat{e}_1,\ldots, \hat{e}_n],
\end{equation}
with $R\in SU(2)$, $\hat{e}_i \in S^1$, $S\in \frak{su}(2)$ and $\hat{\xi}_i=\sqrt{-1}\xi_i\in \sqrt{-1} \R$, for every $i=1,\ldots,n$. Then, using Lemma~\ref{lemma:1}, we have
\begin{align*}
& d_{(p,q)}( [A;e_1,\ldots, e_n]  )  = d_{(p,q)}( [e^{\sqrt{-1} S} ; e^{\sqrt{-1} \hat{\xi}_1},\ldots,   e^{\sqrt{-1} \hat{\xi}_n}] \cdot [R;\hat{e}_1,\ldots, \hat{e}_n] ) \\ & =  d_{(p,q)\cdot  [e^{\sqrt{-1} S} ;  e^{-\xi_1},\ldots,   e^{-\xi_n}]} ([R;\hat{e}_1,\ldots, \hat{e}_n])  +  d_{(p,q)}([e^{\sqrt{-1} S};  e^{\sqrt{-1} \hat{\xi}_1},\ldots,   e^{\sqrt{-1} \hat{\xi}_n}]  ) \\ 
& =  d_{(p,q)} ([e^{\sqrt{-1} S} ; e^{-\xi_1},\ldots,   e^{-\xi_n}]),
\end{align*}
since $R\in SU(2)$ and  $ \hat{e}_i\in S^1$. Hence, $d_{(p,q)}$ induces a function on $K^{\C}/K$.

\begin{lemma}\label{lemma2}
For each $(p,q)\in T^* \C^{2n}$ let  $d_{(p,q)}: K^{\C} \to \R$  be the function in \eqref{def:function}.

Then,  for  $(S , \xi_1,\ldots, \xi_n) \in \mathfrak{su}(2)\oplus \R^n$ and $t\in \R$, we have 
\begin{enumerate}
\item[(1)] 
$ \frac{d}{d t} \left(d_{(p,q)}(  [e^{t \sqrt{-1} S} ; e^{-t \xi_1},\ldots,   e^{- t \xi_n}])\right) = - \sqrt{-1} \langle \mu_{\R} (\hat{p},\hat{q}), (S,\sqrt{-1}\xi_1,\ldots,\sqrt{-1} \xi_n) \rangle$

\noindent with $\hat{q}:=(\hat{q}_1,\ldots,\hat{q}_n)$ and $\hat{p}:=(\hat{p}_1,\ldots,\hat{p}_n)$, where 
$$\hat{q}_i:=e^{-t \sqrt{-1} S}\, q_i \, e^{-t \xi_i}\quad \text{and} \quad \hat{p}_i:=  e^{t \xi_i}\, p_i \,e^{t \sqrt{-1} S}\quad \text{for} \quad i=1,\ldots,n;
$$ 
\item[(2)] $\frac{d^2}{d t^2} \left(d_{(p,q)}(  [e^{t \sqrt{-1} S} ; e^{-t \xi_1},\ldots,   e^{- t \xi_n}])\right) \geq 0$.
\end{enumerate}
\end{lemma}

\begin{proof}
Since $S\in \mathfrak{su}(2)$ is diagonalizable and we can take an orthonormal basis of eigenvectors of $S$, we have $S=A D A^{-1}$ with $A\in SU(2)$ and 
$$
D = \left[ \begin{array}{cc}  - \sqrt{-1} \, x& 0 \\ 0 &  \sqrt{-1}\, x  \end{array}\right],
$$
for some $x>0$ (note that $\text{Tr}(S)=0)$.

Then, writing 
$$
A^{-1} q_i=\left( \begin{array}{l} c_i \\ d_i \end{array} \right), \quad\text{and} \quad p_i A = \left( \begin{array}{ll} a_i & b_i \end{array} \right), 
$$
with $a_i,b_i,c_i,d_i\in \C$, we obtain, for $ [e^{t \sqrt{-1} S} ; e^{-t \xi_1},\ldots,   e^{- t \xi_n}]\in \text{exp} (\sqrt{-1} \frak{k})$, 
\begin{align*}
d_{(p,q)}(  [e^{t \sqrt{-1} S} ; e^{-t \xi_1},\ldots,   e^{- t \xi_n}]) & =  
 \frac{1}{4} \sum_{i=1}^n \left( \lvert \hat{q}_i \rvert^2 - \lvert q_i \rvert^2 + \lvert \hat{p}_i \rvert^2 - \lvert p_i \rvert^2 \right),
\end{align*}
with
$$
\hat{q}_i = A  \left[ \begin{array}{cc}  e^{- t x} & 0 \\ 0 & e^{t x}  \end{array}\right] \,  A^{-1} q_i \, e^{-t \xi_i}= e^{-t \xi_i} A \left( \begin{array}{l} e^{- t x}  c_i \\ e^{t x} \, d_i \end{array} \right)
$$
and 
$$
\hat{p}_i = e^{t \xi_i} \, p_i\,  A  \left[ \begin{array}{cc}   e^{t x} & 0 \\ 0 & e^{- t x}  \end{array}\right] \,  A^{-1} = e^{t \xi_i}  \left( \begin{array}{ll} e^{t x}  a_i & e^{-t x}  b_i \end{array} \right) A^{-1}.
$$
We conclude that
\begin{align}\label{functiondt}\nonumber
& d_{(p,q)}(  [e^{t \sqrt{-1} S} ; e^{-t \xi_1},\ldots,   e^{- t \xi_n}]) \\ & =  \frac{1}{4} \sum_{i=1}^n \left( e^{2t(x+\xi_i)}\lvert a_i \rvert^2 +e^{-2t(x-\xi_i)}\lvert b_i \rvert^2 + e^{-2t(x+\xi_i)} \lvert c_i \rvert^2 + e^{2t(x-\xi_i)}\lvert d_i \rvert^2 \right) + \text{Const}.
\end{align}
On the other hand, by \eqref{real} we have 

$$
\mu_{\R} (\hat{p},\hat{q})\,=\,\frac{1}{2} \sum_{i=1}^n (\hat{q}_i \hat{q}_i^* -\hat{p}_i^* \hat{p}_i )_0 
\oplus \Big(-\frac{1}{2} (\lvert \hat{q}_1\rvert ^2 -\lvert \hat{p}_1\rvert^2), \ldots, 
-\frac{1}{2} (|\hat{q}_n|^2 -|\hat{p}_n|^2) \Big)\, ,
$$
with  
$$
(\hat{q}_i \hat{q}_i^* -\hat{p}_i^* \hat{p}_i )_0 = A  \left[ \begin{array}{cc}   e^{-2t(x+\xi_i)} \lvert c_i \rvert^2 -  e^{2t(x+\xi_i)}\lvert a_i\rvert^2  &  e^{-2t \xi_i}c_i\, \bar{d}_i - \ e^{2t \xi_i}\bar{a}_i\, b_i \\ e^{-2t \xi_i}\bar{c}_i\, d_i - \ e^{2t \xi_i}\, \bar{b}_ia_i & e^{2t(x-\xi_i)} \lvert d_i \rvert^2 -  e^{-2t(x-\xi_i)}\lvert b_i\rvert^2 \end{array}\right]_0 A^{-1}
$$
and 
$$
-\frac{1}{2} (\lvert \hat{q}_i\rvert ^2 -\lvert \hat{p}_i\rvert^2)=-\frac{1}{2} \left(  e^{-2t(x+\xi_i)} \lvert c_i \rvert^2 + e^{2t(x-\xi_i)} \lvert d_i \rvert^2
-  e^{2t(x+\xi_i)}\lvert a_i\rvert^2  -  e^{-2t(x-\xi_i)}\lvert b_i\rvert^2\right).
$$
Hence,
\begin{align}\label{derivative}
\sqrt{-1}&\, \langle \mu_{\R} (\hat{p},\hat{q}),  (S,\sqrt{-1}\xi_1,\ldots, \sqrt{-1}\xi_n) \rangle   =  -\frac{1}{2} \sum_{i=1}^n \left( (x+\xi_i) e^{2t(x+\xi_i)}\lvert a_i \rvert^2 \right. \\ &\hspace{1cm} \left.  - (x-\xi_i) e^{-2t(x-\xi_i)}\lvert b_i \rvert^2 - (x+\xi_i) e^{-2t(x+\xi_i)} \lvert c_i \rvert^2 + (x-\xi_i) e^{2t(x-\xi_i)}\lvert d_i \rvert^2 \right)\nonumber \\ & 
 \hspace{1cm} = - \frac{d}{d t} \left(d_{(p,q)}(  [e^{t \sqrt{-1} S} , e^{-t \xi_1},\ldots,   e^{- t \xi_n}])\right) \nonumber
\end{align}
and  $(1)$ follows.

Differentiating the RHS of \eqref{derivative} we obtain
\begin{align}\label{eq:secder}
\frac{d^2}{d t^2}  \left(d_{(p,q)}(  [e^{t \sqrt{-1} S} ; e^{-t \xi_1},\ldots,   e^{- t \xi_n}])\right) & =
  \sum_{i=1}^n \left( (x+\xi_i)^2 e^{2t(x+\xi_i)}\lvert a_i \rvert^2  + (x-\xi_i)^2 e^{-2t(x-\xi_i)}\lvert b_i \rvert^2  \right. 
\\ & \hspace{.2cm} \left. +\, (x+\xi_i)^2 e^{-2t(x+\xi_i)} \lvert c_i \rvert^2 + (x-\xi_i)^2 e^{2t(x-\xi_i)}\lvert d_i \rvert^2 \right). \nonumber
\end{align}
and $(2)$ follows.

\end{proof}

\begin{lemma}\label{lemma3}
For each $(p,q)\in T^* \C^{2n}$ let  $d_{(p,q)}: K^{\C} \to \R$  be the function in \eqref{def:function}. Then
\begin{enumerate}
\item[(1)] $[A; e_1,\ldots, e_n] \in K^{\C}$ is a critical point of $d_{(p,q)}$ if and only if 
$$(p,q)\cdot [A; e_1,\ldots, e_n] \in \mu_{\R}^{-1}(0,0);$$
\item[(2)] If $(p,q)\in \mu_{\C}^{-1} (0,0)^{\text{st}}$ then $d_{(p,q)}$ induces a strictly convex function on $ K^{\C} /K$.
\end{enumerate}
\end{lemma}

\begin{proof}
If 
$$[A;e_1,\ldots, e_n] = [e^{\sqrt{-1} S} ; e^{\sqrt{-1} \hat{\xi}_1},\ldots,   e^{\sqrt{-1} \hat{\xi}_n}] \cdot [R; \hat{e}_1,\ldots, \hat{e}_n] \in K^{\C},$$
with $R\in SU(2)$, $\hat{e}_i \in S^1$, $S\in \frak{su}(2)$ and $\hat{\xi}_i=\sqrt{-1}\xi_i\in \sqrt{-1} \R$ ($i=1,\ldots,n$)
is a critical point of $d_{(p,q)}$, then 
$$\left(D d_{(p,q)}\right)_{[A; \, e_1,\ldots, e_n]}: T_{[A; \,e_1,\ldots, e_n]} K^{\C} \to \R$$
is the zero map. Let 
$$v=\sqrt{-1} \, (S;  \hat{\xi}_1 ,\ldots, \hat{\xi}_n)\cdot   [A;e_1,\ldots, e_n] \in T_{[A;\,e_1,\ldots, e_n]}K^{\C}$$ 
Then, 
$$
0= \left(D d_{(p,q)}\right)_{[A; \, e_1,\ldots, e_n]} (v) = \frac{d}{dt}\left(d_{(p,q)}\circ c \right)(1),
$$
where $c:[1-\varepsilon,1+\varepsilon]\to K^{\C}$ is given by 
$$
c(t):=  [e^{t \sqrt{-1} S} ; e^{t \sqrt{-1} \hat{\xi}_1},\ldots,   e^{t \sqrt{-1} \hat{\xi}_n}] \cdot [R; \hat{e}_1,\ldots, \hat{e}_n].
$$
(Note that $c(1)=[A;e_1,\ldots, e_n] $ and $\dot{c}(1)=v$.) Then, by Lemma~\ref{lemma:1}
\begin{align*}
0 =  \frac{d}{dt}\left(d_{(p,q)}\circ c \right)(1)  &=  \frac{d}{dt}\left( d_{(p,q)}(    [e^{t \sqrt{-1} S} ; e^{t \sqrt{-1} \hat{\xi}_1},\ldots,   e^{t \sqrt{-1} \hat{\xi}_n}])\right)(1)
\end{align*}
and, by Lemma~\ref{lemma2}, we have that 
$$
(p,q)\cdot [e^{\sqrt{-1} S} ; e^{\sqrt{-1} \hat{\xi}_1},\ldots,   e^{\sqrt{-1} \hat{\xi}_n}] \in \mu_\R^{-1}(0,0).
$$
Since $[R;\hat{e}_1,\ldots, \hat{e}_n]\in K$ and $\mu_\R^{-1}(0,0)$ is $K$-invariant, we conclude that 
$$(p,q)\cdot [A; e_1,\ldots, e_n] \in  \mu_\R^{-1}(0,0).$$

Conversely, if $(p,q)\cdot [A; e_1,\ldots, e_n] \in \mu_{\R}^{-1}(0,0)$, then, using the decomposition 
$$
[A; e_1,\ldots, e_n] = [e^{\sqrt{-1} S} ; e^{\sqrt{-1} \hat{\xi}_1},\ldots,   e^{\sqrt{-1} \hat{\xi}_n}] \cdot [R; \hat{e}_1,\ldots, \hat{e}_n] 
$$ 
in \eqref{polardec}, we have that 
$$(p,q)\cdot  [e^{\sqrt{-1} S} ; e^{\sqrt{-1} \hat{\xi}_1},\ldots,   e^{\sqrt{-1} \hat{\xi}_n}] \in  \mu_{\R}^{-1}(0,0)$$ 
(as $ \mu_{\R}^{-1}(0,0)$, is $K$-invariant). Moreover, for any
$$
v:=(\sqrt{-1}X; \sqrt{-1} \hat{x}_1 ,\ldots, \sqrt{-1}   \hat{x}_n)\cdot   [A;e_1,\ldots, e_n] \in T_{[A;\, e_1,\ldots, e_n]}K^{\C}
$$
we can take the  path $c:[-\varepsilon,\varepsilon]\to K^{\C}$ given by
$$
c(t):= [e^{\sqrt{-1} (S+t\,X)} ; e^{\sqrt{-1}(\hat{\xi}_1+ t\,\hat{x}_1)},\ldots,   e^{\sqrt{-1} (\hat{\xi}_n+ t\,\hat{x}_n)}] \cdot [R;\hat{e}_1,\ldots, \hat{e}_n] 
$$
(note that $c(0)=[A; e_1,\ldots, e_n] $ and $\dot{c}(0)=v$). Then,
\begin{align*}
\left(D d_{(p,q)}\right)_{[A; e_1,\ldots, e_n]} (v) & = \frac{d}{dt}\left(d_{(p,q)}\circ c \right)(0) \\
&  =  \frac{d}{dt}\left( d_{(p,q)}(     [e^{\sqrt{-1} (S+t\,X)} ; e^{\sqrt{-1}(\hat{\xi}_1+ t\,\hat{x}_1)},\ldots,   e^{\sqrt{-1} (\hat{\xi}_n+ t\,\hat{x}_n)}] ) \right)(0)\\
&  =  \frac{d}{dt}\left( d_{(p,q)\cdot [e^{\sqrt{-1} S} ; \, e^{\sqrt{-1}\hat{\xi}_1},\ldots,  e^{\sqrt{-1}\hat{\xi}_n}]}( [e^{t \sqrt{-1} X} ; e^{t \sqrt{-1}\hat{x}_1},\ldots,   e^{t \sqrt{-1} \hat{x}_n}] 
) \right)(0)\\
& = - \sqrt{-1} \langle \mu_{\R} (\hat{p},\hat{q}), (X;\hat{x}_1,\ldots,  \hat{x}_n) \rangle = 0
\end{align*}
since 
$$(\hat{p},\hat{q}) = (p,q)\cdot [e^{\sqrt{-1} S} ; e^{\sqrt{-1}\hat{\xi}_1},\ldots,  e^{\sqrt{-1}\hat{\xi}_n}] \in \mu_{\R}^{-1}(0,0),$$ 
and so $[A; e_1,\ldots, e_n]$ is a critical point of $d_{(p,q)}$.

By Lemma~\ref{lemma2} and \eqref{eq:secder}, if $(p,q)\in \mu_{\C}^{-1} (0,0)^{\text{st}}$ then
$$
\frac{d^2}{d t^2}  \left(d_{(p,q)}(  [e^{t \sqrt{-1} S} , e^{-t \xi_1},\ldots,   e^{- t \xi_n}])\right)   =0 \Leftrightarrow x=\xi_i=0 \,\,  \text{for every $i=1,\ldots,n$},
$$
since  $p_i\neq 0$ and $q_i\neq 0$ for every $i=1,\ldots,n$.
We conclude that, if $(p,q)\in \mu_{\C}^{-1} (0,0)^{\text{st}}$, we have 
$$
\frac{d^2}{d t^2} \left(d_{(p,q)}(  [e^{t \sqrt{-1} S} ; e^{-t \xi_1},\ldots,   e^{- t \xi_n}])\right)  >0  
$$ 
for all $[e^{\sqrt{-1} S};  e^{- \xi_1},\ldots,   e^{-  \xi_n}] \neq [\text{Id}; 1,\ldots,1]$, proving $(2)$.
\end{proof}


We can now describe $X_0^n$ as a finite dimensional GIT quotient:
\begin{theorem}\label{git} Let $\mathcal{P}_0^n$ be the space defined in \eqref{eq:P0n}. Then,
\begin{enumerate}
\item[(1)] $\mathcal{P}_0^n \,\subset\, \mu_{\C}^{-1} (0,0)^{\text{st}}$ and
\item[(2)] there exists a natural bijection 
$$ \iota\,: X_0^n=\mathcal{P}_0^n /K\,\longrightarrow\,
\mu_{\C}^{-1} (0,0)^{\text{st}}/ K^{\C}.$$
\end{enumerate}
\end{theorem}

\begin{rem}
This result  is a consequence of the original Theorem of Kempf and Ness \cite{KN}, further explored by Mumford, Fogarty and Kirwan in \cite[Chapter $8$ and  Appendix $2$ C]{MFK}. Indeed, the function $d_{(p,q)}$ in \eqref{def:function} measures how far ${\bf p}_v$ gets from $\lvert v \rvert^2$, where  $v=(p,q)$, and ${\bf p}_v$ is the function used in \cite{KN} to keep track of the changes in length along the orbit of $v$.  Even though this theorem applies to our spaces, we include a direct, concrete adaptation of the proof  to our context for completeness.
\end{rem}

\begin{proof}

$(1)$ Let $(p,q)\in \mathcal{P}_0^n \subset \mu_{\C}^{-1} (0,0)$. By \eqref{real1}, we have that $p_i, q_i\neq 0$ for all $i\in \{1,\ldots, n\}$ and so the stability condition $(i)$ in Definition~\ref{stability} holds. If $\{1,\ldots, n\}$ is straight at $(p,q)$ then, by  \eqref{complex1}, we can assume that
$$
q_i=\left( \begin{array}{l} c_i \\ 0\end{array} \right), \quad p_i=\left( \begin{array}{ll} 0 & b_i  \end{array} \right), \quad \text{for} \quad i=1,\ldots,n,
$$
with $b_i,c_i\in \C\setminus\{0\}$.
Then, by \eqref{real2} we have 
$$
\sum_{i=1}^n \lvert c_i\rvert ^2= \sum_{i=1}^n \lvert b_i\rvert ^2 = 0,
$$
implying that $p_i=q_i=0$ for every $i$, which is impossible. We conclude that condition $(ii)$ in Definition~\ref{stability} is also satisfied.

$(2)$ Let us first see that $\iota$ is injective. If $\iota([p,q]_\R)=\iota([\hat{p},\hat{q}]_\R)$ for some $[p,q]_\R, [\hat{p},\hat{q}]_\R \in X_0^n$
then $[p,q]_\C= [\hat{p},\hat{q}]_\C$, implying that
$$
(\hat{p},\hat{q})= (p,q) \cdot [A; e_1,\ldots,e_n], 
$$
for some $[A; e_1,\ldots,e_n] \in K^{\C}$. Since $(\hat{p},\hat{q})\in  \mathcal{P}_0^n $ we have that $[A; e_1,\ldots,e_n]$ is a critical point of the function $d_{(p,q)}$ defined in \eqref{def:function} (see Lemma~\ref{lemma3}) and then, as $d_{(p,q)}$ induces a strictly convex function in $K^{\C}/K$, we conclude that $[A; e_1,\ldots,e_n] \in K$, implying that $[p,q]_\R=[\hat{p},\hat{q}]_\R$.

To see that $\iota$ is surjective, take any $(p,q)\in \mu_{\C}^{-1} (0,0)^{\text{st}}$. In order to find $[\hat{p},\hat{q}]_\R$ such that  $\iota([\hat{p},\hat{q}]_\R)=[p,q]_\C$, it is enough to show that $d_{(p,q)}$ has a minimum. Indeed, if that is the case and  $[A; e_1,\ldots,e_n]\in K^{\C}$ is a minimizer, we know, by Lemma~\ref{lemma3}, that
$$
(\hat{p},\hat{q}):=(p,q)\cdot [A; e_1,\ldots,e_n] \in \mu_\R^{-1}(0,0)
$$ 
and then $\iota([\hat{p},\hat{q}]_\R)=[p,q]_\C$. Note that $(\hat{p},\hat{q})\in \mu^{-1}_{HK}((0,0),(0,0))$ and, since  by stability condition $(i)$  we have $\hat{p}_i,\hat{q}_i\neq 0$ for every $i=1,\ldots,n$, we conclude that $(\hat{p},\hat{q})\in  \mathcal{P}_0^n $.

Now to show that $d_{(p,q)}$ has a minimum, we know, by Lemma~\ref{lemma2}, that $d_{(p,n)}$ induces a strictly convex function on $ K^{\C}/K$. Hence, using the polar decomposition of $K^{\C}$ in \eqref{polardec}, it is enough to show that 
\begin{equation}\label{limit}
\lim_{t\to + \infty} d_{(p,q)}( [e^{t \sqrt{-1} S} ; e^{-t \xi_1},\ldots,   e^{- t \xi_n}])= +\infty
\end{equation}
for every nonzero $(S ;  \xi_1, \ldots,   \xi_n)\in \mathfrak{su}(2) \oplus \R^n $. 

If $S=0$ then, since  $(0 ;  \xi_1, \ldots,   \xi_n)\neq 0$, there exists some $i_0\in \{1,\ldots,n\}$ such that $\xi_{i_0}\neq0$. Moreover, since
$$
d_{(p,q)}([e^{t \sqrt{-1} S} ; e^{-t \xi_1},\ldots,   e^{- t \xi_n}]) = \frac{1}{4} \sum_{i=1}^n \left(\lvert e^{t\xi_i} \,p_i\rvert^2 + \lvert e^{-t\xi_i} \,q_i\rvert^2\right) + \text{Const}, 
$$
we have that \eqref{limit} holds since $\xi_{0},p_{i_0},q_{i_0}\neq 0$.

If $S\neq 0$ then, writing $S=A D A^{-1}$ with $A\in SU(2)$ and 
$$
D = \left[ \begin{array}{cc}  - \sqrt{-1} \, x& 0 \\ 0 &  \sqrt{-1}\, x  \end{array}\right],
$$
for some $x>0$, we have
\begin{align*}
& d_{(p,q)}(  [e^{t \sqrt{-1} S} ; e^{-t \xi_1},\ldots,   e^{- t \xi_n}]) \\ & =  \frac{1}{4} \sum_{i=1}^n \left( e^{2t(x+\xi_i)}\lvert a_i \rvert^2 +e^{-2t(x-\xi_i)}\lvert b_i \rvert^2 + e^{-2t(x+\xi_i)} \lvert c_i \rvert^2 + e^{2t(x-\xi_i)}\lvert d_i \rvert^2 \right) + \text{Const}.
\end{align*}
If there exists $i\in \{1,\ldots,n\}$ such that 
\begin{equation}\label{eq:cond}
(a_i\neq 0 \,\, \wedge \,\,\xi_i > -x) \vee (b_i\neq 0  \,\,\wedge \,\, \xi_i > x) \vee (c_i\neq 0  \,\, \wedge  \,\, \xi_i <  -x) \vee (d_i\neq 0  \,\,\wedge \,\, \xi_i < x)
\end{equation}
then \eqref{limit} holds. Since the case where  $x\geq \xi_i \geq -x$ and  $a_i=d_i=0$ is impossible, as we would have  
$$q_i\in \left\langle\left(\begin{array}{c}1 \\ 0 \end{array}\right)\right \rangle\quad \forall i\in\{1,\ldots,n\},$$ 
contradicting stability condition $(ii)$ in Definition~\ref{stability}, the result follows. (Note that if $\xi_i>x$ then \eqref{eq:cond} holds as one of $a_i,b_i$ is necessarily non zero and similarly, if $\xi_i<-x$,  as one of $c_i,d_i$ is non zero.)

\end{proof}

%

%

From Theorem \ref{git}  it follows
that
\begin{equation}\label{eq:nullstablequotient}
X_0^n \,=\, \mu_{\C}^{-1} (0,0)^{\text{st}} / K^{\C}\, .
\end{equation}

We  denote the elements in $\mu_{\C}^{-1} (0,0)^{\text{st}}/ K^{\C}$ by 
$[p,q]_{\text{st}}$, and  by $[p,q]_{\R}$ the elements in $\mathcal{P}_0^n/K$, when we need to make an explicit use of one of the two constructions. In all other cases, we will simply write $[p,q]$ for a null hyperpolygon in $X_0^n$.

\section{Spaces of Quasi-Parabolic Higgs bundles}\label{QPHBs}

Let $\Sigma$ be a compact Riemann surface and consider  a set of  $n$ ordered distinguished marked points $D\, =\,
\{x_1,\ldots,x_n\}$ in $\Sigma$. A \emph{quasi-parabolic vector bundle} over $\Sigma$ at $D$  consists of a holomorphic vector bundle $E$  of rank $r$  over $\Sigma$ and a collection of flags 
\begin{equation}\label{eq:par} 
E_x:=E_{x,1} \supsetneq E_{x,2}\supsetneq \cdots \supsetneq E_{x,s_x} \supsetneq E_{x,s_x+1}=\{0\}
\end{equation}
on the fibers $E_x$ of $E$ at each marked  point $x$ in $D$, where, if $r>1$, 
we have $s_x >1$.
\begin{dfn}
 Let $E$ be a rank-$r$ quasi-parabolic bundle over $\Sigma$  at $D$ with quasi-parabolic structure as in \eqref{eq:par}. The \emph{quasi-parabolic degree} of $E$ is defined as
 $$
 \text{q-par deg}\, (E):= \text{deg}\, (E) + \frac{1}{2}\,\, \lvert D \rvert \, r.
 $$
Moreover, its  \emph{quasi-parabolic slope} is 
  $$
 \text{q-par}\,\mu\, (E):=\frac{ \text{q-par}\, \text{deg}\, (E)}{ \text{rank}\,(E)}.
 $$
\end{dfn}
\begin{rem}
The quasi-parabolic degree coincides with the usual parabolic degree of a parabolic vector bundle when all the parabolic weights at all points of $D$ are equal to $\frac{1}{2}$ (see \cite[Definition 12.4.5]{M}).
\end{rem}
A quasi-parabolic isomorphism $\varphi:E \to F$ between two quasi-parabolic bundles at $D$ with flags of the same type is a bundle isomorphism such that
$$
\varphi(E_{x,i}) = F_{x,i}  \quad \text{for all $i=1,\ldots, s_x$ and $x\in D$}.
$$
\begin{dfn}
Let $E$ be a vector bundle over $\Sigma$ with a quasi-parabolic structure at $D$. A quasi-parabolic bundle $F$ over $\Sigma$ at $D^\prime\subset D$ is a quasi-parabolic subbundle of $E$ if
\begin{itemize} 
\item $F$ is a subbundle of $E$ in the usual sense and
\item for every $x\in D^\prime$, the flag of $F_x$ is a nontrivial subflag of $E_x$, meaning that for every $i=1,\ldots, s_x^F$, there exists a $j\in\{1,\ldots,s_x^E\}$, such that $F_{x,i}=E_{x,j}$.
\end{itemize}
\end{dfn}

\begin{rem}
In the above definition any subbundle of the underlying vector bundle of a quasi-parabolic bundle  is also a quasi-parabolic subbundle since we may assume that $D^\prime=\varnothing$. 
\end{rem}

We also introduce a notion of stability.
\begin{dfn}\label{dfn:stable} A quasi-parabolic bundle $E$ over  $\Sigma$ at $D$  is said to be \emph{stable} if 
$$\text{q-par}\, \mu\, ({E)\,>\,\text{q-par}\, \mu\, (L})$$
for every proper quasi-parabolic subbundle $L$ of $E$.
\end{dfn}

Let us now restrict to the  rank-$2$ case. A quasi-parabolic bundle of rank-$2$ over $\Sigma$ at $D$ is  a rank-$2$ vector bundle $E$ over   $\Sigma$ together with  a collection of complete flags
\begin{equation} \label{eq:rk2qpb}
E_x:=E_{x,1}\supsetneq E_{x,2}\supsetneq\{0\}
\end{equation} 
on the fibers $E_x$ of each marked  point $x\in D$, where $\dim_\C E_{x,2}=1$, for every $x\in D$. Note that
$$
\text{q-par deg}\, (E)= \text{deg} (E) + 2\lvert D\rvert.
$$
The \emph{dual} $E^*$ of a rank-$2$ quasi-parabolic bundle $E$ at $D$ is  the  dual of the holomorphic bundle with the collection of the dual flags at the points of $D$.

If $L$ is a quasi-parabolic line subbundle of $E$ then it has a (trivial) quasi-parabolic structure at the points in the set
\begin{equation}
\label{eq:subbundle}
S_L:=\left\{x\in D:\,\, L_x\cap E_{x,2}\neq \{0\} \right\}.
\end{equation}
Moreover,
$$
 \text{q-par deg}\, (L)= \text{deg} \,(L) +\frac{1}{2}\,\,  \vert S_L \rvert.
 $$
Then, by Definition~\ref{dfn:stable}, $E$ is a stable quasi-parabolic bundle  if and only if 
$$
\text{deg}(E)-2\text{deg}(L)> - \,(\, \lvert D\rvert -\lvert S_L\rvert\, ) 
$$
for every quasi-parabolic line subbundle $L$ of $E$.

\begin{rem}\label{rem:stable}
If, in addition,  $E$ is a trivial vector bundle of rank $2$ over $\Sigma$ with a quasi-parabolic structure at $D$, then the corresponding quasi-parabolic bundle  is stable if and only if
\begin{equation}\label{eq:2stable}
\text{deg}(L) < \frac{1}{2} \,\, \left( \,  \lvert D\rvert -\lvert S_L\rvert \,  \right)
\end{equation}
for every quasi-parabolic line subbundle $L$.
 
If $\Sigma=\C \P^1$ then $\deg(L)\leq 0$ and we conclude that \eqref{eq:2stable} is always satisfied except when $L$ is the trivial bundle and $\lvert S_L \rvert =\lvert D\rvert =n$.
This happens only when all the vector spaces $E_{x,2}$ in \eqref{eq:rk2qpb} are the same for every $x\in D$. 

We conclude that all holomorphically trivial rank-$2$ quasi-parabolic bundles over $\C \P^1$ at $D$ are stable except those for which all the $1$-dimensional flag elements are equal for all $x\in D$.
\end{rem}

%
%
%
%
%
%
%
%
%

Let $K_\Sigma$ denote the holomorphic cotangent bundle of $\Sigma$ and let  ${\mathcal O}_{\Sigma}(D)$ be the 
line bundle on $\Sigma$ defined by the divisor $D$. A rank-$2$ \emph{quasi-parabolic Higgs bundle} over $\Sigma$ at $D$ is a pair ${\bf E} \,:=\, (E, \Phi)$, where $E$ is a  rank-$2$ quasi-parabolic 
vector bundle over $\Sigma$ at $D$ and 
$$\Phi \in H^0(\Sigma, S QPar End(E) \otimes K_\Sigma(D))$$ 
is a \emph{Higgs field} on $E$. Here $S QPar End(E)$ is the 
subsheaf of $End(E)$ formed by \emph{strongly quasi-parabolic endomorphisms} $\varphi\,:\,E 
\,\longrightarrow\, E$,  meaning that 
$$
\varphi(E_{x,1}) \subset E_{x,2} \quad \text{and} \quad \varphi(E_{x,2})=0, \quad \text{for all $x\in D$}.
$$ 
Note that $\Phi$ is a meromorphic endomorphism-valued one-form with 
simple poles along $D$ whose residue at each $x\,\in\, D$ is nilpotent with respect to the 
flag, i.e., 
$$
(\text{Res}_x \Phi) (E_{x,i}) \,\subset\, E_{x,i+1}
$$
for all $i=1,2$ and $x\in D$, where we consider $E_{x,3}=\{0\}$. 

The definition of stability extends to quasi-parabolic Higgs bundles in the following way.

\begin{dfn} A rank-$2$ quasi-parabolic Higgs bundle ${\bf E}= (E,\Phi)$ at $D$ is 
\emph{stable} if
\begin{enumerate}
\item the residue of the Higgs field $\Phi$ at each $x\in D$ is nonzero and
\item for all  quasi-parabolic line subbundles $L\subset E$ which are preserved by $\Phi$ we have
\begin{equation}\label{eq:qpbhiggs}
\text{q-par}\,\mu(E) \,>\, \text{q-par}\, \mu(L).
\end{equation}
\end{enumerate}
\end{dfn}

\begin{rem} \label{rem:stabilityHiggs}
By Remark~\ref{rem:stable} we know that when $\Sigma=\C \P^1$ and the bundle $E$ is trivial, condition \eqref{eq:qpbhiggs} is always satisfied except when  all the flag elements $E_{x,2}$ in \eqref{eq:rk2qpb} are the same for every point $x$ in $D$ and $L$ is the trivial bundle $L=E_{x,2} \times \C$. Such line bundles are always preserved by the Higgs field $\Psi$. We conclude that all holomorphically trivial rank-$2$ quasi-parabolic Higgs bundles over $\C \P^1$ at $D$ with nonzero residues at all points of $D$ are stable except those for which the $1$-dimensional flag elements are all equal.
\end{rem}

If $G$ is a complex reductive  Lie group, a  $G$-\emph{Higgs bundle} over $\Sigma$ is a pair $(P,\Phi)$, where $P$ is a principal $G$-bundle on $\Sigma$ and $\Phi$ is a holomorphic section of $P(\frak{g}) \otimes K$, where $P(\frak{g}):= P\times_\text{Ad}{\frak{g}}$ is  the adjoint bundle associated with $P$  \cite{Hi1}. For $G=GL(2, \C)$ one obtains classical Higgs bundles. Moreover, for matrix groups, this definition can be restated in terms of vector bundles. In particular, to any principal  $SL(2,\C)$-bundle we can associate  a rank-$2$ vector bundle with trivial determinant, using the fundamental $2$-dimensional  representation of $SL(2,\C)$. Conversely, if $E$ is a rank-$2$ vector bundle, its frame bundle is a $GL(2,\C)$-principal bundle and, if in addition its determinant is trivial, the structure group  of this principal bundle can be reduced to $SL(2,\C)$ \cite{HS}. Consequently, we can think of $SL(2,\C)$-Higgs bundles as pairs $(E,\Phi)$, where $E$ is a rank-$2$ holomorphic bundle over $\Sigma$ with trivial determinant, and $\Phi$ is  a traceless holomorphic section of $End(E)\otimes K$ \cite{G-P2}. We can further generalize this definition and consider rank-$2$ quasi-parabolic $SL(2,\C)$-Higgs bundles over $\Sigma$.

\begin{dfn}
A rank-$2$  quasi-parabolic $SL(2,\C)$-Higgs bundle over $\Sigma$ at $D$ is a rank-$2$ quasi-parabolic Higgs bundle ${\bf E}= (E,\Phi)$ over $\Sigma$ at $D$, where the  underlying vector bundle of $E$ has trivial determinant and the Higgs field $\Phi$ is traceless.
\end{dfn}


Let $\mathcal{H}_0^n$ be the moduli space of stable rank-$2$ quasi-parabolic $SL(2,\C)$-Higgs bundles  over $\Sigma=\C \P^1$ at $D$,
for which the underlying holomorphic vector bundle is 
trivial.  In Appendix~\ref{appendix} we  define  a structure of 
complex manifold on  $\mathcal{H}_0^n$  which turns it into a moduli space
of quasi-parabolic Higgs bundles in the usual sense in algebraic/complex 
geometry.

We will now see that it is isomorphic to 
the space of null hyperpolygons $X_0^n$  in \eqref{eq:nullstablequotient}. The correspondence between these two spaces is given by the map
\begin{equation}
\label{eq:isom}
\begin{array}{rl}
\mathcal{I}:X_0^n \longrightarrow & \mathcal{H}_0^n\\ \\
{[p,q]}_{ \textnormal{st}} \longmapsto & [E_{(p,q)}\, , {\Phi}_{(p,q)}] =: 
\mathbf [{\bf E}_{(p,q)}] \\
\end{array}
\end{equation}
where $E_{(p,q)}$ is the trivial vector bundle $\C\P^1\times \C^2
\,\longrightarrow\, \C\P^1$ with the quasi-parabolic structure given by the flags 
\begin{align*}\label{flag}
 \C^2 & \supset \langle q_i\rangle \supset \{0\}
\end{align*}
over the $n$ marked points in $D=\{x_1,\ldots,x_n\} \,\subset\, \C \P^1$ and 
$$\Phi_{(p,q)} \,\in\, 
H^0\big(\C \P^1, S QPar End(E_{(p,q)}) \otimes K_{\C\P^1}(D)\big)$$
is the Higgs field uniquely determined by the 
residues:
\begin{equation}
\label{eq:res}
\text{Res}_{x_i} \Phi_{(p,q)} \,:=\, q_i \, p_i
\end{equation}
for each $x_i\,\in\, D$ (see the proof of Theorem~\ref{thm:main} below for a justification that $\Phi_{(p,q)}$ is uniquely determined by its residues).

\begin{theorem}\label{thm:main}
The spaces $X_0^n$ and $\mathcal{H}_0^n$ are isomorphic.
\end{theorem}

\begin{proof}
We first show that the map $\mathcal I$ is well-defined: the Higgs field $\Phi_{(p,q)}$ is uniquely defined,  the QPHB $\mathbf E_{(p,q)}$ is stable, and the map $\mathcal I$ is independent of the choice of representative in ${[p,q]}_{\textnormal{st}}$.

\begin{itemize}
\item Let $[p,q]\in X_0^n$. Since $(p,q)\in \mu_\C^{-1}(0,0)$ we have by \eqref{eq:complex0}  and \eqref{complex1} that 
$$
\sum_{i=1}^n \, q_i \, p_i = \sum_{i=1}^n \,( q_i \, p_i)_0 = 0. 
$$
Consequently the residues in \eqref{eq:res} uniquely determine 
the  meromorphic endomorphism valued $1$-form $\Phi_{(p,q)}$ on $\C \P^1$.

\item By the stability of $(p,q)$ we have that the set $\{1,\ldots,n\}$ is not straight at $(p,q)$ and then, by Remark~\ref{rem:stable}, the quasi-parabolic Higgs bundle is stable.

\item To show that $\mathcal I$ is independent of the choice of a representative in ${[p,q]}_{\textnormal{st}}$, let $(\tilde p , \tilde q )$ be  in the $K^\C$-orbit of $(p,q)$ and consider $[E_{(\tilde p, \tilde q)}, \Phi_{(\tilde p, \tilde q)} ]$ as before. The Higgs field $\Phi_{(\tilde p, \tilde q)}$ is defined by the residues
\begin{align*}
\text{Res}_{x_i}  \Phi_{(\tilde p_i, \tilde q_i)} := \tilde{q}_i \, \tilde{p}_i = A^{-1} q_i z_i z_i^{-1} p_i A =
A^{-1} (q_i p_i) A = A^{-1} \, \text{Res}_{x_i} \Phi_{(p,q)} A
\end{align*}
for some $A \in SL(2, \C)$ and $z_i \in \C^*$. Moreover, the flags in $E_{(\tilde p, \tilde q)}$ are determined by $\tilde q_i = A^{-1} q_i z_i$. 
Since
 $q_i z_i \in \langle q_i \rangle$, we have that $A^{-1} \langle q_i\rangle =\langle \tilde{q}_i \rangle$ for every $i\in \{1,\ldots, n\}$ and so 
 $[E_{(p,q)}, \Phi_{(p,q)}] = [E_{(\tilde p, \tilde q)}, \Phi_{(\tilde p, \tilde q)} ]$. 

\end{itemize}

Let us now see  that $\mathcal{I}$ is injective. Let  $[p,q]_{\text{st}},[\hat{p},\hat{q}]_{\text{st}}\in X_0^n$  be such that 
$$\mathcal{I}([p,q]_{\text{st}})=\mathcal{I}([\hat{p},\hat{q}]_{\text{st}}).$$ 
Then there exists $A\in SL(2,\C)$ such that $\langle q_i \rangle = A \langle \hat{q}_i \rangle$ for every $i=1,\ldots,n$, and so
$$
\hat{q} = q \cdot [A; e_1,\ldots,e_n]
$$
for some $[A;e_1,\ldots,e_n]\in K^\C$. 
Moreover,
$$
\text{Res}_{x_i} \Phi_{(\hat{p},\hat{q})} = A^{-1}\, \text{Res}_{x_i} \Phi_{(p,q)} \, A 
$$
and so
$$
A^{-1} q_ip_i A = \hat{q}_i\hat{p}_i  = A^{-1} q_i \, e_i \, \hat{p}_i, 
$$
implying that $q_i p_i A = q_i \,\hat{p}_i\, e_i$. Multiplying both sides by $q^*_i$, we obtain
$$
q_i^* q_i \, p_i A = q_i^* q_i \, \hat{p}_i e_i \Leftrightarrow \lvert q_i \rvert^2 p_i A =  \lvert q_i \rvert^2 \hat{p}_i e_i,
$$
and, since $\lvert q_i \rvert\neq 0$, we get
$$
\hat{p}_i  = e_i^{-1} p_i A.
$$
We conclude that 
$$
(\hat{p},\hat{q}) =(p, q) \cdot [A;e_1,\ldots,e_n]
$$
for some $[A;e_1,\ldots,e_n]  \in K^\C$ and so  $[p,q]_{\text{st}}=[\hat{p},\hat{q}]_{\text{st}}$.

To show that $\mathcal{I}$ is surjective, let $[E,\Phi]\in  \mathcal{H}_0^n$. For each  point $x_i \in D, $ let  $q_i = (c_i, d_i)^t$ be a generator of  
$E_{x_i,2}$ and, representing the residue of the Higgs field $\Phi$ at $x_i$ by a traceless matrix
\begin{equation*}
 N_i:= \text{Res}_{x_i} \Phi,  
\end{equation*}
we take  $p_i $ to be 
\begin{equation}
\label{eq:pi}
 p_i = (a_i, b_i) := \frac{1}{\lvert q_i\rvert^2} \cdot q_i^* N_i.
\end{equation}
Note that  $\lvert q_i\rvert \neq 0$  for all $i\in \{1,\ldots,n\}$ since the flags are nontrivial.

The pair $(p,q)$ constructed in this way is in $\mu_{\C}^{-1}(0,0)$ and is stable: 

\begin{itemize}  
\item Since $N_i$ is nilpotent with respect to the flag, we have that 
$$
p_i q_i = \frac{1}{\lvert q_i\rvert^2} \cdot q_i^* N_i q_i = 0.
$$
Moreover, since $N_i$ is traceless and $ N_i q_i=0$, we get
\begin{equation}\label{eq:product}
q_i \, p_i = N_i = (q_i p_i)_0,
\end{equation}
and then, since  
the sum of the residues $N_i$ is  $0$, we obtain $\sum_{i=1}^n (q_i p_i)_0=0$. We conclude that
$(p,q) \in \mu_{\C}^{-1}(0,0)$; 

\item To show that $(p,q)$ is stable, we need to check conditions $(i)$ and $(ii)$ of Definition~\ref{stability}. The first one ($q_i \neq 0$ and $p_i\neq 0$ for all $i$) is trivially verified since the flags are complete and, by assumption, the residue of the   Higgs field never vanishes on  $D$ (note that $q_ip_i=N_i$). For $(ii)$, if $S= \{ 1, \ldots, n\}$ were straight at $(p,q)$, then all the flag elements $E_{x,2}$ for $x\in D$ would be the same complex line $\ell$. We could therefore consider the trivial line bundle $L$ over $\C \P^1$ with total space $\C \P^1\times \ell$ which is trivially a quasi-parabolic subbundle of $E$   at $D$ (i.e. $S_L=D$, where $S_L$ is the set  defined in \eqref{eq:subbundle}). Moreover, it is also preserved by the Higgs field $\Phi$ 
since
$$
(\text{Res}_{x_i} \Phi)\, ( \ell) = 0
$$
and then, by Remark~\ref{rem:stable}, the quasi-parabolic Higgs bundle $\mathbf E$ would not be stable.

\end{itemize}
It is easy to check that $\mathcal{I}([p,q]_{\text{st}})=[E,\Phi]$ and so $\mathcal{I}$ is surjective.

From the local description of  $\mathcal{H}_0^n$  given in Proposition~\ref{prop:appprop1} it is easy to check that the map $\mathcal{I}$ is a biholomorphism between the two complex manifolds (where we consider in $X_0^n$ the complex structure  induced by the complex structure $I$).

\end{proof}

\section{Involutions}\label{sec:inv}

%

The compact real form $\frak{su}(2)$ of $\frak{sl}(2,\C)$ is the fixed point set of the anti-linear  involution
\begin{align*}
\iota_1: \frak{sl}(2,\C) & \to \frak{sl}(2,\C) \\
X & \mapsto -X^*
\end{align*}
and the involution
\begin{align*}
\iota_2: \frak{su}(2) & \to \frak{su}(2)  \\
X & \mapsto \overline{X}
\end{align*}
gives the Cartan decomposition 
\begin{equation}\label{eq:cartan}
\frak{su}(2) = \frak{k} \oplus \frak{h},
\end{equation} where
$$
\frak{k} :=\{ X\in \frak{su}(2): \iota_2(X)=X\} = \frak{so}(2)
$$
and 
$$
 \frak{h} :=\{ X\in \frak{su}(2): \iota_2(X)=- X\} = \{ \text{$2\times2$  traceless, symmetric imaginary matrices} \}.
$$
Then, denoting by $\frak{m}$ the Lie algebra $\sqrt{-1}\frak{h}$ of traceless symmetric endomorphisms of $\R^2$, the  split real form $\frak{k}\oplus \frak{m}$  of $\frak{sl}(2,\C)$ corresponding to \eqref{eq:cartan} is $\frak{sl}(2,\R)$. It is the fixed point set of the extension of the involution $\iota_2$ to $\frak{sl}(2,\C)$.

At the group level, the compact real form $SU(2)$ of $SL(2,\C)$ is the fixed point set of the anti-holomorphic involution
\begin{align*}
\hat{\iota}_1: SL(2,\C) & \to SL(2,\C)  \\
X & \mapsto (X^*)^{-1}
\end{align*}
while  the anti-holomorphic involution 
\begin{align*}
\hat{\iota}_2: SL(2,\C) & \to SL(2,\C)  \\
X & \mapsto \overline{X}
\end{align*}
has fixed point set the split real form $SL(2,\R)$ of $SL(2,\C)$.

The involution $\theta:=\iota_2 \circ \iota_1$ given by $\theta(X)=-X^t$ is the Cartan involution of the split real form $SL(2,\R)$ of $SL(2,\C)$ (see \cite[Chapter 5]{OV}). In the  moduli space
$\mathcal{H}_0^n$ of quasi-parabolic Higgs bundles, the map $\theta$ induces the involution 
\begin{equation}\label{eq:invPHB}
(E, \Phi) \,\stackrel{\Theta}{\longmapsto}\, (E^*,\Phi^t).
\end{equation}
 
Recall that a   $SL(2,\R)$-\emph{Higgs bundle} over $\Sigma$ is a pair $(P,\Phi)$, where $P$ is a holomorphic principal $SO(2,\C)$-bundle over $\Sigma$ (i.e. a line bundle $L$) and $\Phi$ is a holomorphic section of $P(\frak{m}^\C) \otimes K$, where $P(\frak{m}^\C):= P\times_\text{Ad}\frak{m}^\C=L^2\oplus L^{-2}$ is  the adjoint bundle associated with $P$ (see for example \cite{Mu}). Indeed, considering the standard complex structure 
$$
I=\left(\begin{array}{cc} 0 & -1 \\ 1 & 0 \end{array} \right) \in End(\R^2)
$$
we can take the decomposition of $\R^2\otimes \C \simeq \C^2 = V_+\oplus V_{-}$, where $V_+$ and  $V_{-}$ are the eigenspaces $V_+:=\langle (1,-\sqrt{-1})\rangle$ and $V_{-}:=\langle (1,\sqrt{-1})\rangle $ of $I$ (corresponding to the eigenvalues $\pm \sqrt{-1}$).
Taking the basis $\{v_+,v_{-}\}$ of $\C^2$ with $v_+:= (1,-\sqrt{-1}) \in V_+$ and $v_{-}:=\frac{1}{2} (1,\sqrt{-1}) \in V_-$,  a traceless symmetric endomorphism $\psi$  of $\R^2\otimes \C$ for the standard Euclidean structure (denoted by $\cdot$) on $\R^2$ (i.e. an element of  
$\frak{m}^\C$) can be written in this basis as a matrix 
$$
A:= \left(\begin{array}{cc} a_{11} & a_{12} \\ a_{21} & -a_{11} \end{array} \right)
$$
with 
$$
a_{11}=(a_{11} v_+ + a_{21} v_{-})\cdot v_{-} =  v_+ \cdot (a_{12} v_+ -a_{11}v_{-})= -a_{11}, 
$$
since $\psi( v_+) \cdot v_{-} = v_+ \cdot \psi( v_{-})$. Hence
$$
A= \left( \begin{array}{cc} 0 & a_{12} \\ a_{21} & 0 \end{array} \right).
$$
Moreover, the adjoint action of $SO(2,\C)\simeq \C^*$ on $\frak{m}^\C$ is given by
$$
\lambda \cdot \psi =  \left( \begin{array}{cc}  \lambda^{-1} & 0 \\ 0 & \lambda  \end{array} \right) \left(\begin{array}{cc} 0 & a_{12} \\ a_{21} & 0 \end{array} \right)   \left( \begin{array}{cc}  \lambda & 0 \\ 0 & \lambda^{-1}  \end{array} \right) = \left( \begin{array}{cc} 0 & \lambda^{-2} a_{12} \\ \lambda^2 a_{21} & 0 \end{array} \right). 
$$
Consequently, we can think of an $SL(2,\R)$-\emph{Higgs bundle} over $\Sigma$ as a Higgs bundle $(E,\Phi)$ whose underlying vector bundle is $E=L^2\oplus L^{-2}$ and the Higgs field $\Phi$ is 
$$\Phi=(\Phi_+,\Phi_{-})\in H^0(\Sigma, L^2\otimes K)\oplus H^0(\Sigma, L^{-2}\otimes K).$$


Following \cite{BGM} we can again generalize this definition  to the quasi-parabolic case.

\begin{definition}\label{def:sl2RpHb}
A rank-$2$  quasi-parabolic $SL(2,\R)$-Higgs bundle over $\Sigma$ at $D$ is a rank-$2$ quasi-parabolic Higgs bundle ${\bf E}= (E,\Phi)$ over $\Sigma$ at $D$, such that
\begin{itemize}
\item the  underlying vector bundle is $E=L\oplus L^{-1}$ for some line bundle $L$ over $\Sigma$ with quasi-parabolic structure at $D$ given by
$$
E_x:=E_{x,1}\supset E_{x,2}\supset\{0\},
$$
with $E_{x,2}=L_x$ for $x$ in some subset $S_L\subset D$ and  $E_{x,2}=L^{-1}_x$ for $x\in D\setminus S_L$;
\item  the Higgs field  is of the form
$$
\Phi=\left(\begin{array}{cc} 0 & \alpha \\ \beta & 0 \end{array} \right):E\to   E \otimes K_\Sigma(D)
$$
with $\beta = 0$ on the fibers over $S_L$ and $\alpha=0$ on those of  $D\setminus S_L$.
\end{itemize}
\end{definition}

We will see that  the fixed-point set of the involution  $\Theta: \mathcal{H}_0^n \to \mathcal{H}_0^n$ defined in \eqref{eq:invPHB} is the moduli space of  rank-$2$  quasi-parabolic $SL(2,\R)$-Higgs bundles over $\C \P^1$ at $D$ for which the underlying vector bundle is holomorphically trivial (i.e. those for which the line bundle $L$   in Definition~\ref{def:sl2RpHb} is the trivial line bundle over $\C \P^1$).

%


%
%
%


%

For that we will use the isomorphism in
\eqref{eq:isom} and study the fixed point set of the corresponding involution 
on the  space $X_0^n$ of null hyperpolygons, that is  the map
\begin{align}
\label{eq:inv}
X_0^n & \stackrel{\iota}{\longrightarrow} X_0^n \\ \nonumber
[p,q] & \longmapsto\, [q^t,p^t].
\end{align}
It is holomorphic with respect to the complex structure induced by $I$ and anti-holomorphic with respect to those induced by $J$ and $K$, where $I,J,K$ are the standard hyperk\"ahler complex structures on $T^*\C^{2n}$.

\begin{rem}
The dual of a rank-$2$ trivial quasi-parabolic bundle $E$ with quasi-parabolic structure given by a collection of complete flags 
\begin{equation*}
\C^2 \simeq E_x:=E_{x,1}\supsetneq E_{x,2}\supsetneq E_{x,3}=\{0\}
\end{equation*} 
on the fibers $E_x$ at the parabolic points $x\in D$ is the  rank-$2$ trivial quasi-parabolic bundle $E$ with quasi-parabolic structure given by the collection of dual flags
\begin{equation*}
\C^2 \simeq E^*_x:=E_{x,3}^0 \supsetneq E_{x,2}^0\supsetneq E_{x,1}^0=\{0\},
\end{equation*} 
where 
$$
E_{x,i}^0 = \{f \in E_x^*:\,\, f(E_{x,i})=0 \}.
$$
Then, if $[p,q]_{st}\in X_0^n$ is such that
$$
\mathcal{I}^{-1}([E, \Phi]) = [p,q]_{st},
$$
where $\mathcal{I}:X_0^n \to  \mathcal{H}_0^n$ is the map defined in \eqref{eq:isom}, we have
$$
E_{x_i,2} = \langle q_i\rangle 
$$
for every $i=1,\ldots,\lvert D \rvert$ and so, by \eqref{eq:complex0}, we have
$$
E_{x_i,2}^0 = \langle p_i^t \rangle. 
$$
Moreover,  from \eqref{eq:product} we have 
$$N_i^t=p_i^t \, q_i^t$$
and so
$$
q_i^t = \frac{1}{\lvert p_i\rvert^2} \cdot \overline{p}_i \, N_i^t = \frac{1}{\lvert p_i^t\rvert^2} \cdot (p_i^t)^* \, N_i^t .
$$
Hence
$$
\mathcal{I}^{-1}([E^*, \Phi^t]) = [q^t,p^t]_{st}
$$
and the involution $\Theta:  \mathcal{H}_0^n \to \mathcal{H}_0^n$ defined in \eqref{eq:invPHB} induces the map 
$\iota: X_0^n\to X_0^n$ on  \eqref{eq:inv}.

Note that this map is well-defined. Indeed, if $[\widetilde{p},\widetilde{q}\,]=[p,q]\in X_0^n$, then there exists $[A;e_1,\ldots,e_n]\in K$
such that
$$
(\widetilde{p},\widetilde{q}\,)=(p,q)\cdot [A;e_1,\ldots,e_n]
$$  
and so 
$$
\widetilde{p}_i ^{\,\,t}= (e_i^{-1} p_i A)^t = A^T p_i^t e_i^{-1} \quad \text{and} \quad \widetilde{q}_i^{\,\,t} = (A^{-1} q_i e_i)^t=e_i q_i^{\,t} \,\bar{A}, 
$$
implying that
$$
( \widetilde{q}^{\,\,t},\widetilde{p} ^{\,\,t})= (q^{\,t},p^{\,t})\cdot [\,\bar{A}; e_1^{-1},\ldots, e_n^{-1}],
$$
and so $[\widetilde{q}^{\,\,t},\widetilde{p} ^{\,\,t}\,]=[q^{\,t},p^{\,t}]$.
Note that, since $A\in SU(2)$, we have 
$$
(\bar{A})^*\,\,\bar{A}= \overline{A^*A}=I
$$
and $(\bar{A})^{-1}=A^T$.
\end{rem}



Let $\mathcal S$ be the collection of subsets  $S \subset \{1,\ldots, n\}$ such that $2\leq \lvert S\rvert \leq n-2$ and $1\in S$. For $S\in  \mathcal S$, consider the sets
\begin{equation} \label{eq:ZS}
Z_S \,:= \,\big\{[p,q] \in X_0^n:\,~ S \, \text{ and } S^c\, 
\text{ are straight  at }\, (p,q) \big\}.
\end{equation}
Then we have the following result.

\begin{theorem}
\label{thm:invfixedpoints}
The fixed-point set of the involution in \eqref{eq:inv} is
$$ (X_0^n)^{\,\iota}\,:=\bigcup_{{\tiny \begin{array}{c} S \in\, \mathcal S \end{array}}} 
Z_S\, ,$$
where $Z_S$ is defined above.
Moreover, each $Z_S$ is a non-compact symplectic manifold of dimension $2(n-3)$
and $(X_0^n)^{\,\iota}$ has $2^{\,n-1}-(n+1)$ connected components.
\end{theorem}

\begin{proof}

Suppose that $[p,q]\,\in\, X_0^n$ is a fixed point of $\iota$. Then there 
exists 
an element $$[A;e_1,\cdots,e_n]\,\in\, K\setminus \{Id\}$$ such that
$$
e_i^{-1} p_i A \,=\, q_i^{\,t} \quad \text{and}\quad A^{-1}q_i e_i \,=\,p_i^{\,t}, 
\,\,\text{for} \,\, i\,=\,1,\ldots, n\,.
$$
In particular,
\begin{equation}\label{eq:fp0}
q_i= A p_i^{\, t} e_i^{-1} \quad \text{and} \quad  (A\bar{A})^T p_i^{\,t}=p_i^{\,t},  
\end{equation}
and so $ p_i^{\,t} $ is an eigenvector of $(A\bar{A})^T$ associated to the eigenvalue $1$, for all $i=1,\ldots,n$. Since $(A\bar{A})^T\in SU(2)$, we conclude that 
$$
A\bar{A}= \rm{Id}.
$$
Writing 
$$
A\,=\,\left( \begin{array}{rc} \alpha & \beta \\ -\bar{\beta} & \bar{\alpha} \end{array} \right),
$$
with $\alpha,\beta\in \C$ such that $\lvert \alpha\rvert^2+\lvert \beta\rvert^2=1$,
we have
$$
{\rm Id} =A\,\bar{A} \,=\,\left( \begin{array}{rc} \alpha & \beta \\ -\bar{\beta} & \bar{\alpha} \end{array}  \right) \left( \begin{array}{rc} \bar{\alpha} & \bar{\beta} \\ - \beta & \alpha \end{array}  \right)  = \left( \begin{array}{rc} \lvert \alpha\rvert^2 - \beta^2 & \alpha \bar{\beta} + \alpha \beta \\ -( \overline{\alpha \beta} + \bar{\alpha} \beta)  & \lvert \alpha\rvert^2 -\bar{\beta}^2\end{array} \right),
$$
and so
\begin{equation}\label{eq:fp2}
 A\,=\,\left( \begin{array}{rc} \alpha & k \sqrt{-1}\\ k \sqrt{-1} & \bar{\alpha} \end{array} \right), 
\end{equation}
for some $\alpha \in \C$ and $k\in \R$ such that $\lvert{\alpha}\rvert^2 + k^2 =1$.

Assuming, without loss of generality, that  
$$ 
p_1 = \left( \begin{array}{ll} 0 & b \end{array}\right) \quad \text{and}\quad q_1= \left(\begin{array}{c} c \\ 0 \end{array} \right) 
$$
with $b,c\in \C\setminus\{0\}$ such that $\lvert b\rvert=\lvert c\rvert$, we have from \eqref{eq:fp0} and \eqref{eq:fp2} that 
$$
\left(\begin{array}{c} c \\ 0 \end{array} \right)  = \left( \begin{array}{rc} \alpha & k \sqrt{-1}\\ k \sqrt{-1} & \bar{\alpha} \end{array} \right) \left(\begin{array}{c} 0 \\ b \end{array} \right) e_1^{-1}
$$
and so $\alpha=0$, impliying that 
\begin{equation}\label{eq:fp3}
 A\,=\,\left( \begin{array}{rc}0 & \pm \sqrt{-1}\\ \pm \sqrt{-1} & 0 \end{array} \right). 
\end{equation}
Then, writing, as usual,
$$
q_i= \left(\begin{array}{c} c_i \\ d_i \end{array} \right) \quad \text{and}\quad p_i = \left( \begin{array}{ll} a_i & b_i \end{array}\right), \quad \text{for $i=2,\ldots, n$,}
$$
with $(a_i,b_i),(c_i,d_i)\in \C^2\setminus\{0\}$, we have, from \eqref{eq:fp0},
$$
 \left(\begin{array}{c} c_i \\ d_i \end{array} \right)= \left( \begin{array}{rc}0 & \pm \sqrt{-1}\\ \pm \sqrt{-1} & 0 \end{array} \right)  \left(\begin{array}{c} a_i \\ b_i \end{array} \right) e_i^{-1 } = \left(\begin{array}{c} \pm b_i e_i^{-1}\sqrt{-1} \\ \pm a_i e_i^{-1} \sqrt{-1}\end{array} \right).
$$ 
Then, by \eqref{complex1}, we conclude that $a_ib_i=0$ and so
there exists an index set 
$S\subset \{1,\ldots, n \}$ with $1 \in S$ such that
\begin{align} \label{eq:Z0} 
p_i & = \left( \begin{array}{ll} 0 & b_i \end{array}\right), \quad q_i= \left(\begin{array}{c} c_i \\ 0 \end{array} \right), \forall i\in S \\ \nonumber 
p_i & = \left( \begin{array}{ll} a_i & 0 \end{array}\right), \quad q_i= \left(\begin{array}{c} 0 \\ d_i \end{array} \right),\forall i \in S^c.
\end{align}
Since $\lvert q_i\rvert = \lvert p_i\rvert$, we conclude that
\begin{equation}
\label{eq:Z}
\lvert c_i\rvert = \lvert b_i\rvert\,\,\text{for all $i\in S$} \quad \text{and} \quad \lvert d_i\rvert=  \lvert a_i\rvert \,\, \text{for all $i\in S^c$}.
\end{equation}


Moreover, since $\sum_{i=1}^n ( q_i q_i^* - p_i^*p_i)_0 =0$, we obtain by \eqref{real2} that
\begin{equation}\label{eq_Z2}
\sum_{i\in S} |c_i|^2 = \sum_{i\in S^c}  |a_i|^2. 
\end{equation}
On the other hand, since $\sum_{i=1}^n (q_ip_i)_0 \,=\,0$, we have by \eqref{complex2} that 
\begin{equation}\label{eq:zero}
\sum_{i \in S} b_i c_i \,=\,\sum_{i \in S^c} a_i d_i \,=\, 0\, .
\end{equation}
Since by \eqref{eq:Z0}  we have $b_i,c_i\neq 0$ for all $i\in S$ (as $p_i,q_i\neq 0$) and 
$1\in S$, 
it follows from \eqref{eq:zero}  that  $S$ has cardinality at least two. On the other hand, since $\{1,\ldots,n\}$ is not straight at $(p,q)$ (by stability condition $(ii)$ in Definition~\ref{stability}), we have that $S^c\neq \varnothing$. Since by \eqref{eq:Z0}  we have $a_i,d_i\neq 0$ for all $i\in S$ (as $p_i,q_i\neq 0$), it follows from \eqref{eq:zero}  that $S^c$ has cardinality at least two. Ws conclude that $S\in \mathcal{S}$ and $[p,q]\in Z_S$.

Conversely, given any subset $S\in \mathcal{S}$, we have that all the points in $Z_S$ are fixed by $\iota$. Indeed, if $[p,q]\in Z_S$, then $S$ and $S^c$ are straight at $(p,q)$ and so we can assume that
\begin{align*}
p_i & = \left( \begin{array}{ll} 0 & b_i \end{array}\right), \quad q_i= \left(\begin{array}{c} c_i \\ 0 \end{array} \right), \forall i\in S \\ \nonumber 
p_i & = \left( \begin{array}{ll} a_i & 0 \end{array}\right), \quad q_i= \left(\begin{array}{c} 0 \\ d_i \end{array} \right),\forall i \in S^c.
\end{align*}
Then 
$$
\iota[p,q]=[q^t,p^t]=[(p,q)\cdot [A;e_1,\ldots,e_n]]=[p,q],
$$ 
with 
$$
A=\,\left( \begin{array}{rc}0 & \sqrt{-1}\\  \sqrt{-1} & 0 \end{array} \right)\in SU(2) 
$$
and 
$$
e_i= \frac{b_i}{c_i}\sqrt{-1}\in S^1 \quad \text{if $i\in S$}\quad \text{and}\quad e_i= \frac{a_i}{d_i}\sqrt{-1} \quad \text{if $i\in S^c$}.
$$

Finally, since for every subset $S\subset \{1,\ldots,n\}$, we have that either $1\in S$ or $1\in S^c$ , the number of subsets in $\mathcal{S}$  is 
$$
\frac{1}{2} \sum_{k=2}^{n-2} \left( \begin{array}{l} n \\ k \end{array} \right) = 2^{(n-1)}-(n+1).
$$
\end{proof}
%
%
%

Using the isomorphism in \eqref{eq:isom} we conclude that the fixed point set of the involution \eqref{eq:invPHB} on the moduli space of quasi-parabolic $SL(2,\C)$-Higgs bundles $\mathcal{H}_0^n$ is formed by the  quasi-parabolic $SL(2,\R)$-Higgs bundles in $\mathcal{H}_0^n$.

\begin{theorem}\label{thm:0.1}
The fixed-point set of the involution $\Theta:\mathcal{H}_0^n\to \mathcal{H}_0^n$ in \eqref{eq:invPHB} on the space of 
quasi-parabolic $SL(2,\C)$-Higgs bundles $\mathcal{H}_0^n$ is
$$ \mathcal{H}_{0,n}^{\R}\,: =\, \bigcup_{{\tiny \begin{array}{c}1\in S \subset\{1,\ldots,n\} \\ 2\leq \lvert S\rvert \leq n-2\end{array}}} \mathcal{Z}_S$$
where
$\mathcal{Z}_S\,\subset\, \mathcal{H}_0^n$ is formed by the quasi-parabolic $SL(2,\R)$-Higgs 
bundles $\mathbf{E}\,=\,(E,\Phi)\in  \mathcal{H}_0^n$ such that
\begin{enumerate}
\item[(i)] the quasi-parabolic vector bundle $E$ admits 
a direct sum decomposition 
$$E\,=\,L_0\oplus L_1,$$
where $L_0$ and $L_1$ are trivial line bundles over $\C \P^1$
and the quasi-parabolic structure at $D$ is given by
$$
E_x:=E_{x,1}\supset E_{x,2}\supset\{0\},
$$
with $E_{x,2}=(L_0)_x$ for $x\in \{x_i\in D: \, i\in S\}\subset D$ and  $E_{x,2}=(L_{1})_x$ otherwise;
\item  the Higgs field  is of the form
$$
\Phi=\left(\begin{array}{cc} 0 & \alpha \\ \beta & 0 \end{array} \right):E\to   E \otimes K_{\C \P^1}(D)
$$
with $\beta = 0$ on the fibers over $x_i\in D$ such that $i\in S$ and $\alpha=0$ on those over $x_i$ with $i\in S^c$.
(i.e. it is either upper or lower triangular with respect to the above decomposition, according to whether $i$ is in $S$ or in $S^c$). 
\end{enumerate}
Moreover each $\mathcal{Z}_S$ is a non-compact manifold of dimension $2(n-3)$ and $\mathcal{H}_{0,n}^{\R}$ has 
$$2^{(n-1)}-(n+1)$$ 
connected components.
\end{theorem}


\section{Null Polygons in Minkowski $3$-space}\label{sec:NullMink}

Let $\R^{2,1}$ be the \emph{Minkowski $3$-space}, consisting of $\R^3$ equipped with the signature 
$(-,-,+)$-inner product given by
$$
v \circ w \,=\, -v_1w_1-v_2w_2+v_3w_3,
$$
The \emph{Minkowski norm} of a vector $v\in 
\R^{2,1}$ is 
$$
\lvert\lvert v\rvert\rvert_{2,1}\,=\, \sqrt{\lvert v\circ v \rvert}\ 
$$
and the vectors $v\in \R^{2,1}$ are classified according to the sign of $v \circ v 
$. Those $v$ for which  $v\circ v =0$ are called  \emph{light-like} or \emph{null} vectors and form the \emph{light cone}. If $ v\circ v \,>\, 0$, then 
$v$ is a \emph{time-like} vector and, if $v\circ v \,<\,0$,  it is 
a \emph{space-like} vector. A null vector is said to be lying in the \emph{future}
(respectively \emph{past}) light cone if $v_3>0$ (respectively $v_3<0$).

The \emph{Minkowski cross product} $\dot{\times}$ on $\R^{2,1}$ is defined as 
$$
v\dot{\times} w\, :=\, \det \left( \begin{array}{rrr} -e_1 & -e_2 & e_3 \\ v_1 & 
v_2 & v_3 \\ w_1 & w_2 & w_3 \end{array}\right),
$$
for  $v=(v_1,v_2,v_3), w=(w_1,w_2,w_3)$ and $e_1,e_2,e_3$ the standard 
unit vectors in $\R^3$ and the Lie algebra $(\R^3,\dot{\times})$ is isomorphic to 
$\mathfrak{su}(1,1)$ via the map
$$
\left(\begin{array}{l} x \\ y \\ t \end{array} \right)\,\longmapsto\,\frac{1}{2} 
\left( \begin{array}{cc} -\sqrt{-1} t & x + \sqrt{-1} y \\ x - \sqrt{-1} y & 
\sqrt{-1} t \end{array} \right)\, .
$$
Note that  $\mathfrak{su}(1,1)^*\simeq \sqrt{-1} \cdot \mathfrak{su}(1,1)$ is also identified with $\R^3$.

Under this identification, the Minkowski inner product $(\circ)$ corresponds to the pairing between $\mathfrak{su}(1,1)^*$ and  $\mathfrak{su}(1,1)$ given by
$$(A,B)\, \longmapsto\, 2 \sqrt{-1}\,\cdot  \mathrm{trace}(AB),$$ 
for $A\in \mathfrak{su}(1,1)^*$ and  $B\in \mathfrak{su}(1,1)$.

The coadjoint action of (the non-compact group)  ${\rm SU}(1,1)$  on $\frak{su}(1,1)^* \simeq \R^{2,1}$
is defined by
$$
(A \cdot u) \circ v = u \circ \left( A^{-1} v A \right)
$$ 
for every  $v\in \frak{su}(1,1)$, where $A \in SU(1,1)$ and $u\in \frak{su}(1,1)^*$. In particular, if
$$A=  \left(\begin{array}{rr} \alpha & \beta  \\\bar{\beta} & \bar{\alpha}  \end{array} \right) \in SU(1,1)$$ 
(with $\alpha,\beta \in \C$  such that $\lvert \alpha\rvert^2-\lvert \beta \rvert^2=1$), 
then $A\cdot u$ corresponds to multiplication by the $3\times 3$-matrix
\begin{equation}\label{eq:matrix} 
A_{\alpha,\beta}:= \left(\begin{array}{rrr} \rm{Re}\, (\alpha^2 -\beta^2) & -\rm{Im} \,(\alpha^2+\beta^2) & -2\, \rm{Im} \,(\alpha\beta) \\
\rm{Im}\, (\alpha^2 -\beta^2) & \rm{Re} \,(\alpha^2+\beta^2) & 2\, \rm{Re} \,(\alpha\beta) \\
2\, \rm{Im} \,(\alpha\bar{\beta}) & 2 \,\rm{Re} \,(\alpha\bar{\beta}) & \lvert \alpha\rvert^2 +\lvert \beta\rvert^2
\end{array} \right).
\end{equation}

The ${\rm SU}(1,1)$ coadjoint orbits in $\R^{2,1}$ are subsets of the surfaces 
$$-x^2-y^2+t^2\,=\,R,\quad R\in \R.$$ 
If $R>0$ one obtains a two-sheeted hyperboloid and each sheet is a coadjoint orbit.
If $R<0$ one obtains a one-sheeted hyperboloid which is also a coadjoint orbit.
If $R=0$ one obtains three distinct coadjoint orbits: the origin and the two components  $C^\pm$ of the light cone defined by
$$
t= \pm \sqrt{x^2+y^2}.
$$
We will call $C^\pm$ the \emph{future} and \emph{past light cones}. 

Each coadjoint orbit of $SU(1,1)$ is a symplectic manifold having  an invariant symplectic structure (the Kostant--Kirillov
form).  Moreover, the action of $SU(1,1)$ is Hamiltonian and the corresponding moment map is the inclusion map to $\frak{su}(1,1)^*$.

We will study the geometry of the symplectic 
quotients of  products of several copies of future and past light cones $C^\pm$
with respect to the diagonal ${\rm SU}(1,1)$-action.
For that, let us fix two positive integers $k_1,k_2$ with 
$k_1+k_2\,=\,n$. We will consider \emph{null polygons} in Minkowski $3$-space that have the 
first $k_1$ edges in the past light cone and the last $k_2$ edges in the future
light cone. A 
closed null
polygon is one whose sum of the first $k_1$ sides in the past light cone is symmetric to the sum of the last $k_2$ sides in the future light cone. The space  of all such closed null polygons can be identified with the zero level set of the moment map
\begin{equation}
\label{eq:mmhyp}
\begin{array}{ccc}
\mu: \mathcal{O}_1\times \cdots \times \mathcal{O}_n & \longrightarrow & 
\mathfrak{su}(1,1)^* \\
(u_1, \ldots, u_n) & \longmapsto & u_1 + \cdots + u_n
\end{array}
\end{equation}
for the diagonal ${\rm SU}(1,1)$--action, where $\mathcal{O}_i\cong 
C^-$ 
is the past light cone if $1 \leq i \leq k_1$, and 
$\mathcal{O}_i\cong C^+$ is the future light cone if $k_1+1 \leq i \leq 
n$, equipped with the Kostant-Kirillov symplectic form on coadjoint orbits. 

Let $\left(\mathcal{O}_1\times \cdots \times \mathcal{O}_n \right)^{\text{reg}}$ be the set of regular points of $\mu$ in $\mathcal{O}_1\times \cdots \times \mathcal{O}_n$. By equivariance,  it is invariant under $SU(1,1)$ and it is easy to check that the action of $SU(1,1)$ on $\left(\mathcal{O}_1\times \cdots \times \mathcal{O}_n \right)^{\text{reg}}$ is free (cf. Remark~\ref{rem:regular_null}). If 
$$\mu^{-1}(0)_{\text{reg}}:=\left(\mathcal{O}_1\times \cdots \times \mathcal{O}_n \right)^{\text{reg}}\cap \mu^{-1}(0)$$
is nonempty then it is a smooth submanifold of $\left(\mathcal{O}_1\times \cdots \times \mathcal{O}_n \right)^{\text{reg}}$ (by the Implicit Function Theorem). Moreover, if $\omega$ is the symplectic form  on $\mathcal{O}_1\times \cdots \times \mathcal{O}_n$ and $x\in \mu^{-1}(0)_{\text{reg}}$, then the kernel of $\omega_x$ on  the tangent space
$$
T_x(\mu^{-1}(0)_{\text{reg}})=\ker (d\mu)_x= \left(T_x (SU(1,1)\cdot x)\right)^{\omega_x}
$$
is
$$
T_x(\mu^{-1}(0)_{\text{reg}})\cap \left(T_x(\mu^{-1}(0)_{\text{reg}}\right)^{\omega_x}  =\{X_x\in T_x(\mathcal{O}_1\times \cdots \times \mathcal{O}_n):\, X\in \frak{su}(1,1)\} .
$$
Hence, $i^* \omega$, where  $i:\mu^{-1}(0)_{\text{reg}}\to \mathcal{O}_1\times \cdots \times \mathcal{O}_n$ is the inclusion map, is a closed $2$-form on $\mu^{-1}(0)_{\text{reg}}$ and the leaves of its null foliation are the orbits of $SU(1,1)$. Consequently, the space
$$
M^{k_1,k_2}_0\,:=\,\mu^{-1}(0)_{\text{reg}}/{\rm SU}(1,1)\, 
$$
is a symplectic manifold.

\begin{rem}\label{rem:regular_null}
A point $u=(u_1,\ldots,u_n)\in  \mathcal{O}_1\times \cdots \times \mathcal{O}_n$  is regular if and only if its $SU(1,1)$-stabilizer is trivial. Indeed, if it exists $A\in SU(1,1)$ such that $A\neq Id$ and $Au_i=u_i$ for  $i=1,\ldots,n$ then $\dim \langle u_1,\ldots, u_n\rangle = 1$ (meaning that the null polygon lies along a line) and the stabilizer of $u$ is $1$-dimensional. 

To see this,
%
%
let us consider a matrix $A_{\alpha,\beta}$ as in \eqref{eq:matrix}  fixing a vector $u\in C^\pm$. Using a rotation around the $t$-axis we can assume, without loss of generality, that $u=(0,1,\pm1)$. Then, up to conjugation by a rotation around the $t$-axis, we have   that $A_{\alpha,\beta}$ is of the form
$$
 A_{\alpha,k\sqrt{-1}}:= \left(\begin{array}{rcr}1 & \pm 2 k & -2 k \\
\mp 2 k & 1 - 2\, k^2 & \pm 2\,  k^2\\
-2 \,k & \mp 2 \,k^2  &1+2 \, k^2
\end{array} \right)
$$ 
with $k\in \R$ and so the stabilizer of $u$ has dimension $1$. Moreover, 
the space of eigenvectors of $ A_{\alpha,k\sqrt{-1}}$ associated with the eigenvalue $1$ is always one-dimensional for every $k\neq 0$.

%
\end{rem}
We conclude that for $n\geq 4$ and $k_1,k_2\geq 2$ with $k_1+k_2=n$ the space $M^{k_1,k_2}_0$ is nonempty and  its points represent the null polygons  that do not lie along a line. Moreover, the spaces $M^{k_1,k_2}_0$ and $M^{k_2,k_1}_0$ are symplectomorphic. 
Note that if $k_1=1$ or $k_2=1$ the space is empty since the sum of two noncolinear past null vectors is a past timelike vector and the sum of a  past timelike vector with a past null vector is still a past timelike vector.
%
%

\begin{theorem}
The space $M^{k_1,k_2}_0$ is non-compact.
\end{theorem}

\begin{proof}

Let $[v]=[v_1,\ldots,v_n]$ by a polygon in $M^{k_1,k_2}_0$, where $[v]=[v_1,\ldots,v_n]$ with $v_i\in C^-$ for $i=1,\ldots, k_1$ and $v_i\in C^+$ otherwise. Note that the vectors $v_1,\ldots, v_{k_1}$ are not all aligned since, by definition, $(v_1, \ldots, v_n)$ is a regular value of the moment map $\mu$ defined in \eqref{eq:mmhyp} (cf. Remark~\ref{rem:regular_null}). Therefore,  $w:=v_1+\cdots+ v_{k_1}$ is a time-like vector\footnote{The sum of two noncolinear past null vectors is a past timelike vector and the sum of a  past timelike vector with a past null vector is still a past timelike vector.} and we can use a rotation around the $t$-axis followed by a \emph{boost}
$$
T_\phi :=\left(\begin{array}{ccc} 1 & 0 & 0 \\ 0 & \cosh{\phi} & \sinh{\phi} \\ 0 & \sinh{\phi} & \cosh{\phi} \end{array} \right) \in SU(1,1)
$$
along the $y$-direction to place the vector $w$ along the $t$-axis. Let $\ell$ be the Minkowski length of $w$. 
Note that $\ell$ never vanishes and can take any value in $(0,\infty)$. Indeed if we take, for instance, 
$$
v_1:=(0,m,-m), \quad v_2=\cdots=v_{k_1}=\left(0,-\frac{m}{k_1-1},-\frac{m}{k_1-1}\right)
$$ 
and
$$
v_{k_1+1}:=(0,m,m), \quad v_{k_1+2}=\cdots=v_{n}=\left(0,-\frac{m}{k_1-1}, \frac{m}{k_1-1}\right)
$$
with $m\in \mathbb{N}$, then $w= v_1+\cdots +v_{k_1}=(0,0,-2m)$ and $\lvert\lvert w\rvert\rvert_{2,1} = 2m$. 

We can therefore consider the (continuous) map $\ell:M^{k_1,k_2}_0\to (0,\infty)$ which for each $[v]\in M^{k_1,k_2}_0$ gives the Minkowski length of the vector $v_1+\cdots+v_{k_1}$ and the result follows. Note that $\ell$ is the moment map of the circle action obtained by rotating the section of the polygon formed by the first $k_1$ edges around the $t$-axis. This map has no critical values so all level sets are diffeomorphic and we obtain a diffeomorphism between $M^{k_1,k_2}_0$ and $P\times (0,\infty)$ where $P$ is a level set of $\ell$.

\end{proof}


%

\section{Null hyperpolygons, quasi-parabolic $SL(2,\R)$-Higgs bundles and null polygons in $\R^{2,1}$}
\label{sec:Mink}

Let $[p,q]\in Z_S$ be a fixed point of the involution \eqref{eq:inv} in the null hyperpolygon space, where $Z_S$ is defined in \eqref{eq:ZS} and consider the vectors $u_i\in \R^3$ given by
$$
u_i \,:=\, \frac{1}{2} (p^*_ip_i -q_i q^*_i)_0 +\frac{\sqrt{(-1)}}{2} (p_i^*q_i^* +q_i p_i)_0\, ,
$$
where we use identifications $\mathfrak{su}(2)^*\cong( \R^3)^* 
\cong \mathfrak{su}(1,1)^*$. Assuming, without loss of generality, that 
$$p_i=\left( \begin{array}{ll}0 & b_i\end{array}\right),\quad  q_i=\left( \begin{array}{l} c_i \\ 0 \end{array}\right)$$ 
with $\lvert b_i\rvert=\lvert c_i\rvert$, for $i \in S$ and
$$p_i=\left( \begin{array}{ll} a_i& 0 \end{array}\right),\quad  q_i=\left( \begin{array}{l} 0 \\ d_i \end{array}\right)$$ 
with $\lvert a_i\rvert=\lvert d_i\rvert$, for $i \in S^c$,
 we have
\begin{equation}
\label{eq:uS}
u_i\,=\,\left( \operatorname{Re}{(b_i c_i)}\, , \operatorname{Im}{(b_i 
c_i)}\, ,-\lvert c_i \rvert^2 \right),\quad \text{for}\quad i\in S
\end{equation}
and
\begin{equation}
\label{eq:uSc}
u_i=\left( \operatorname{Re}{(a_i d_i)},- \operatorname{Im}{(a_i d_i)},\lvert a_i\rvert^2\right),\quad \text{for}\quad i\in S^c.
\end{equation}
Note that for $i\in S$ we have
$$
 u_i\circ u_i \,= - \lvert b_i\, c_i\rvert^2 + \lvert c_i \rvert^4= \lvert c_i\rvert^2 (\lvert c_i\rvert^2-\lvert b_i\rvert^2) = 0,
$$
and, for $i\in S^c$,
$$
 u_i\circ u_i \,= - \lvert a_i\, d_i\rvert^2 + \lvert a_i \rvert^4= \lvert a_i\rvert^2 (\lvert a_i\rvert^2-\lvert d_i\rvert^2) = 0,
$$
and so  the vectors $u_i$  are null vectors in $\R^{2,1}$.
By \eqref{eq_Z2} and \eqref{eq:zero}, we have that
$$
\sum_{i=1}^n u_i\,=\,0,\,
$$
and so the vectors $u_i$ form a closed null polygon in Minkowski $3$-space with 
$\lvert S\rvert$ sides in the past light cone and  $n-\lvert S\rvert$ sides in the future light cone.

Moreover, the vectors $u_i$ are not all aligned. Indeed, if the vectors $u_i$ were colinear for $i\in S$, there would exist $i_0\in S$ and $k_i\in \R^+$ such that
$$
u_i=k_i u_{i_0} \quad \text{for  all $i\in S$.}
$$ 
Then
$$
\sum_{i\in S} b_ic_i= b_{i_0}c_{i_0} \sum_{i\in S} k_i \neq 0,
$$
contradicting \eqref{eq:zero}. We conclude that the vectors $u_i$  form a regular null polygon in $\R^{2,1}$ (cf. Remark~\ref{rem:regular_null}).

\begin{theorem}\label{thm:polMink}
For any $S\,\subset \, \{1,\ldots, n\}$, such that $1\in S$ and $2\leq \lvert S \rvert\leq  n-2$, the components $\mathcal{Z}_S$ 
and $Z_S$, of the fixed-point sets of the involutions  \eqref{eq:invPHB} and 
\eqref{eq:inv} in $\mathcal{H}_0^n$ and $X_0^n$ respectively, are diffeomorphic to the moduli space 
$$
M^{\lvert S \rvert,\lvert S^c \rvert}_0
$$
of closed null polygons in Minkowski $3$-space.
\end{theorem}

\begin{proof} Let  $S\,\subset \, \{1,\ldots, n\}$ be such that $1\in S$ and $2\leq \lvert S \rvert\leq  n-2$. After a suitable reshuffling, 
consider the map $\varphi:Z_S \longrightarrow M^{\lvert S\rvert,\lvert S^c 
\rvert}_0$ defined by \eqref{eq:uS} and \eqref{eq:uSc}. Hence $\varphi([p,q])$ is the element of 
$M^{\lvert S\rvert,\lvert S^c \rvert}_0$ represented by the null polygon in $\R^{2,1}$ whose first $\lvert S \rvert$ 
sides are the vectors $u_i$ given by \eqref{eq:uS}, and the last $n-\lvert S \rvert$ sides  are the vectors $u_i$ given by  \eqref{eq:uSc}. 

%
%

Let us first see that $\varphi$ is well defined. For that, consider two 
representatives $(p,q)$ and $(p^\prime,q^\prime)$ of the same class $[p,q]$ in 
$Z_S$. Then there exists $[A;e_1,\ldots,e_n]\,\in\, K$ such that
$$
e_i^{-1}p_i A= p_i ^\prime \quad \text{and} \quad A^{-1} q_i e_i =q_i^{\prime}, 
\quad i=1,\ldots,n.
$$
We can assume, without loss of generality, that $p_i=\left( \begin{array}{ll} a_i & b_i\end{array}\right)$, $p_i^\prime=\left(\begin{array}{ll}a_i^\prime & b_i^\prime\end{array}\right)$ with $a_i=a_i^\prime=0$ for $i \in S$ and $b_i=b_i^\prime=0$ for $i \in S^c$, while 
$q_i=\left( \begin{array}{ll} c_i & d_i \end{array}\right)^t$, $q_i^\prime=\left( \begin{array}{ll} c_i^\prime & d_i^\prime \end{array}\right)^t$
with $d_i\,=\,d_i^\prime\,=\,0$ for $i \in S$ and $c_i\,=\,c_i^\prime=0$ for $i 
\,\in\, S^c$. Hence,
$$
A= \left(\begin{array}{cc}\alpha & 0 \\ 0 & \bar{\alpha} \end{array}\right)
$$
for some $\alpha=e^{\sqrt{-1}\,\theta_0} \in S^1$. Then we have
$$
\left(\begin{array}{l} \operatorname{Re}{(b^\prime_i c^\prime_i)} \\ \\ 
\operatorname{Im}{(b^\prime_i c^\prime_i)} \\ \\ -\lvert c^\prime_i \rvert^2 \end{array}\right) = A_{-2\theta_0} \left(\begin{array}{l}\operatorname{Re}{(b_i c_i)} \\ \\ \operatorname{Im}{(b_i c_i)} \\ \\ -\lvert c_i \rvert^2 \end{array}\right)
\,\,\text{for $i \in S$,}
$$
and
$$
\left(\begin{array}{c} \operatorname{Re}{(a^\prime_i d^\prime_i)} \\ \\ - \operatorname{Im}{(a^\prime_i d^\prime_i)} \\ \\ \lvert a^\prime_i \rvert^2 \end{array}\right) = A_{-2\theta_o} \left( 
\begin{array}{c} \operatorname{Re}{(a_i d_i)} \\ \\ - \operatorname{Im}{(a_i d_i)} \\ \\  \lvert a_i\rvert^2 \end{array} \right)\,\, \text{for $i \in S^c$},
$$
where 
$$ A_{-2\theta_0} = \left(\begin{array}{ccc} \cos{2\theta} & \sin{2\theta} & 0 \\ -\sin{2\theta} & \cos{2\theta} & 0 \\ 0 &0 & 1\end{array} \right) $$ is a rotation around the $t$-axis  (an element of  ${\rm SU}(1,1)$). 

To see that $\varphi$ is injective, let  $[p,q], 
[p^\prime,q^\prime]$ be  two points in $Z_S$ with
$$\varphi([p,q])\,=\,\varphi([p^\prime,q^\prime]).$$ 
Then, without loss of generality,  we can write
$$ p_i=\left( \begin{array}{ll}0 & b_i\end{array}\right), \,\, p_i^\prime=\left( \begin{array}{ll}0 & b_i^\prime \end{array}\right)\quad \text{and} \quad q_i=\left( \begin{array}{l} c_i \\ 0 \end{array}\right),\,\, q_i^\prime=\left( \begin{array}{l} c_i^\prime \\ 0 \end{array}\right), 
$$
for $i= 1,\ldots, \lvert S\rvert $ with 
\begin{equation}\label{eq:zerosum}
\sum_{i=1}^{\lvert S\rvert} b_i c_i=\sum_{i=1}^{\lvert S\rvert}  b^\prime_i c^\prime_i=0
\end{equation} 
(cf. \eqref{eq:zero}) and 
$$ p_i=\left( \begin{array}{ll} a_i & 0 \end{array}\right), \,\, p_i^\prime=\left( \begin{array}{ll} a_i^\prime & 0 \end{array}\right)\quad \text{and} \quad q_i=\left( \begin{array}{l} 0 \\ d_i \end{array}\right),\,\, q_i^\prime=\left( \begin{array}{l} 0 \\ d_i^\prime \end{array}\right), \,\, 
$$
for $i= \lvert S\rvert +1,\ldots,n$,  with 
$$\sum_{i=\lvert S\rvert+1}^n a_i d_i=\sum_{i=\lvert S\rvert +1}^n a^\prime_i d^\prime_i=0,$$ 
and there exists $A_{\alpha,\beta}\in SU(1,1)$ as in \eqref{eq:matrix}  such that
\begin{equation}\label{eq:prime1}
\left(\begin{array}{l} \operatorname{Re}{(b^\prime_i c^\prime_i)} \\ \\ \operatorname{Im}{(b^\prime_i c^\prime_i)} \\ \\ -\lvert c^\prime_i \rvert^2\end{array}\right) = A_{\alpha,\beta} \left(\begin{array}{l}\operatorname{Re}{(b_i c_i)} \\ \\ \operatorname{Im}{(b_i c_i)} \\ \\ -\lvert c_i \rvert^2 \end{array}\right)
\,\,\text{for $i= 1,\ldots, \lvert S\rvert$}
\end{equation}
and
\begin{equation}\label{eq:prime2}
\left(\begin{array}{c} \operatorname{Re}{(a^\prime_i d^\prime_i)} \\ \\ - \operatorname{Im}{(a^\prime_i d^\prime_i)} \\ \\  \lvert d^\prime_i \rvert^2 \end{array}\right) = A_{\alpha,\beta} \left( 
\begin{array}{c} \operatorname{Re}{(a_i d_i)} \\ \\ - \operatorname{Im}{(a_i d_i)} \\ \\ \lvert d_i \rvert^2 \end{array} \right)\,\, \text{for $i= \lvert S \rvert +1,\ldots,n$ }.
\end{equation}

Then, we have 
$$
 b_i^\prime c_i^\prime = - 2 \sqrt{-1}\alpha \beta \lvert c_i \rvert^2 + \left(\alpha^2-\beta^2\right) \, \operatorname{Re}\, (b_ic_i) + \sqrt{-1} (\alpha^2+\beta^2) \,\operatorname{Im}(b_ic_i),
$$
for $i=1,\ldots, \lvert S\rvert$   and so, by \eqref{eq:zerosum},
we conclude that $\alpha \beta=0$, implying that $\beta=0$ (since $\lvert \alpha\rvert^2-\lvert \beta\rvert^2=1$).

Consequently,
$$A_{\alpha,\beta}= A_{\alpha,0} = \left(\begin{array}{ccc} \rm{Re}\, \alpha^2 & -\rm{Im} \,\alpha^2 &  0\\
\rm{Im}\, \alpha^2  & \rm{Re} \,\alpha^2 & 0 \\
0 & 0 & 1
\end{array} \right), \quad \text{with $\lvert \alpha\rvert =1$,}
$$ 
is a rotation around the $t$-axis.

Hence
$$
b_i^\prime c_i^\prime \,= \,e^{\sqrt{-1}\, \theta_0} b_i c_i, \quad \text{and} \quad 
\lvert b_i^\prime \rvert^2 = \lvert c_i^\prime \rvert^2 = \lvert b_i \rvert^2 = \lvert c_i\rvert^2, \quad \text{for $i=1,\ldots, \lvert S \rvert$}, 
$$
for some $\theta_0\in [0,2\pi)$ and, by \eqref{eq:prime2}, we have
$$
a_i^\prime d_i^\prime \,=\, e^{-\sqrt{-1}\, \theta_0} a_i d_i, \quad \text{and} \quad 
\lvert a_i^\prime \rvert^2 = \lvert d_i^\prime \rvert^2 = \lvert a_i \rvert^2 = \lvert d_i\rvert^2, \quad \text{for $i=\lvert S\rvert + 1, \ldots, n$}.
$$
Moreover, since 
$$
\vert c_i^\prime\rvert = \lvert c_i\rvert=\vert b_i^\prime\rvert = \lvert b_i\rvert, \quad  \text{for} \quad i=1,\ldots, \lvert S \rvert,
$$ 
and  
$$
\vert a_i^\prime\rvert = \lvert a_i\rvert=\vert d_i^\prime\rvert = \lvert d_i\rvert,\quad  \text{for} \quad i= \lvert S \rvert+1,\ldots,n,
$$ 
we have
$$
c_i^\prime =  e^{\sqrt{-1}\, \gamma_i}c_i , \quad b_i^\prime =  e^{\sqrt{-1}\, (\theta_0-\gamma_i)}b_i \quad \text{for} \quad i=1,\ldots, \lvert S \rvert
$$
and 
$$
a_i^\prime =  e^{\sqrt{-1}\, \phi_i}a_i, \quad d_i^\prime =  e^{-\sqrt{-1}\, (\theta_0+\phi_i)} d_i \quad \text{for} \quad i=1,\ldots, \lvert S \rvert
$$
for some $\gamma_i,\phi_i\in [0,2\pi)$.

We conclude that
$$
p_i^\prime \,= \,p_i A \, e_i^{-1} \quad \text{and} \quad q_i^\prime = A^{-1}q_i \, e_i, \quad 
i\,=\,1,\cdots,n
$$
with $A\,= \left( \begin{array}{cc} e^{-\sqrt{-1}\, \theta_0/2 } & 0 \\ 0 & e^{\sqrt{-1}\, 
\theta_0/2 } \end{array}\right)$, and 
$$
e_i := \left\{ \begin{array}{l} e^{\sqrt{-1}(\gamma_i-\theta_0/2)}, \,\, \text{if} \,\, i=1, \ldots, \lvert S\rvert \\  \\ e^{-\sqrt{-1}(\phi_i+\theta_0/2)}, \,\, \text{if} \,\, i= \lvert S\rvert +1,\ldots, n\end{array}\right.,
$$
implying that $[p,q]=[p^\prime,q^\prime]$.

To  show that $\varphi$ is surjective, let us take 
$[v]\in M^{\lvert S\rvert,\lvert S^c \rvert}_0$, where $[v]=[v_1,\ldots,v_n]$ with $v_i\in C^-$ for $i=1,\ldots, \lvert S\rvert$ and $v_i\in C^+$ otherwise. Note that the vectors $v_1,\ldots, v_{\lvert S\rvert}$ are not all aligned since, by definition, $(v_1, \ldots, v_n)$ is a regular value of the moment map $\mu$ defined in \eqref{eq:mmhyp} (cf. Remark~\ref{rem:regular_null}). Therefore,  $w:=v_1+\cdots+ v_{\lvert S\rvert}$ is a time-like vector
and we can use a rotation around the $t$-axis followed by a \emph{boost}
along the $y$-direction to place the vector $w$ along the $t$-axis. Hence, we can assume that $[v]$ is  represented by a polygon with the first $\lvert S \rvert$ sides past null vectors $(x_i,y_i,- t_i)$  and the last $n-\lvert S\rvert$ sides future null vectors $(x_i,y_i,t_i)$ with $t_i=\sqrt{x_i^2 + y_i^2}$, satisfying
$$
\sum_{i=1}^{\lvert S\rvert} x_i=\sum_{i=1}^{\lvert S\rvert} y_i =\sum_{i=\lvert S\rvert+1}^n x_i = \sum_{i=\lvert S\rvert+1}^n y_i=0.
$$
%
Then $[v]$ is the image of the hyperpolygon $[p,q]$, with
$$
p_i= \left( \begin{array}{ll}0 & \frac{1}{l_i}(x_i+\sqrt{-1}\, y_i)\end{array}\right), \quad q_i= \left(\begin{array}{l} l_i \\ 0 \end{array}\right) \,\, \text{for $i\in S$},
$$
and 
$$ 
p_i= \left( \begin{array}{ll} \frac{1}{l_i} (x_i-\sqrt{-1} y_i) & 0 \end{array}\right), \quad q_i=\left( \begin{array}{c} 0 \\ l_i \end{array}\right)\,\, \text{for $i\in S^c$},
$$
with 
$$
l_i \,=\, \sqrt{t_i},\quad i\,=\,1,\ldots, n\, .
$$
Note that $[p,q]\in Z_S\subset X_0^n$, since 
$$
\sum_{i\in S} b_i c_i \,=\,\sum_{i=1}^{\lvert S\rvert} (x_i + \sqrt{-1}\, y_i )\,= \,0,
\quad  \sum_{i \in S^c} a_i d_i= \sum_{i=\lvert S\rvert +1}^n ( x_i - \sqrt{-1}\, y_i )=0
$$
and
$$
\lvert b_i \rvert^2 =t_i =\lvert c_i\rvert^2 \quad \text{for all $i \in S$}, \quad
\lvert a_i\rvert^2  = t_i = \lvert d_i\rvert^2, \,\,\text{for all $i \in S^c$} ,
$$
where as usual we write $p_i= \left( \begin{array}{ll} a_i& b_i 
\end{array}\right)$ and $q_i= \left( \begin{array}{ll} c_i& d_i \end{array}\right)^t$, for $i=1,\ldots,n$. 

Note that clearly $\varphi$ and its inverse are differentiable and the result
follows.
\end{proof}
%
%
%
%

\section{An Example}\label{sec:ex}

As an example, we consider the case where $n=4$. 


Let $\mathcal{H}_0^4$ be 
the moduli space of quasi-parabolic $SL(2,\C)$-Higgs bundles $(E,\Phi)$ of rank two over $\C \P^1$ at $D$ with $\lvert D\rvert=4$, where the underlying holomorphic vector bundle $E$ is trivial and let us consider the space of null hyperpolygons $X_0^4$.


A subset $S$ of $\{1,2,3,4\}$ such that $\lvert S \rvert \geq 2$ and  $\lvert S^c \rvert \geq 2$ must have exactly two elements.  If, in addition $1\in S$, there are 
exactly three possibilities. Let us denote these sets by $S_1$, $S_2$ and 
$S_3$, where $S_j:=\{1,1+j\}$. 

Then the fixed point set of the involution in \eqref{eq:inv} defined in the space of null hyperpolygons $X_0^4$ has 
exactly three connected components $Z_{S_1},Z_{S_2}$ and $Z_{S_3}$ (cf. Theorem~\ref{thm:invfixedpoints}).

By Theorem~\ref{thm:polMink} we know that each component  $Z_{S_i}$
is diffeomorphic to $M^{2,2}_0$, formed by classes of 
closed polygons in Minkowski $3$-space with the first two sides $u_1, u_2$ in 
the past light cone and the last two, namely $u_3$ and $u_4$, in the future light cone. Let us  consider the diagonal vector
$w=u_1+u_2$ connecting the origin to the third vertex of the polygon. Since the vectors $u_1$ and $u_2$ are not aligned,
$w$  is a past time-like vector and we can consider its Minkowski 
length $\ell$. Using a  rotation around the $t$-axis followed by a \emph{boost} we can assume that $w$ lies along the $t$-axis.

For each value of $\ell\in (0,\infty)$ we have a circle of possible classes of polygons obtained by 
rotating the last two sides of the polygon around the diagonal, while fixing the other two. The length $\ell$ of $w$
is the 
moment map for the bending flow obtained by this rotation of the last two sides of the polygon around the diagonal
with a constant angular speed while fixing the other two vectors. Hence, 
$Z_{S_i}$ is a non-compact toric manifold of dimension $2$ with moment map $\ell$. Moreover, $\ell$ has no critical values in $(0,\infty)$ so $Z_{S_i}$ is diffeomorphic to $\C\setminus\{0\}$. 

We conclude that  the fixed point sets  $(X_0^4)^{\iota}$ and $\mathcal{H}_{0,4}^\R$ (the space of quasi-parabolic $SL(2,\R)$-Higgs bundles over $\C \P^1$ with four quasi-parabolic points) of the involutions \eqref{eq:inv} and \eqref{eq:invPHB} have three non-compact components diffeomorphic to $\C\setminus\{0\}$.

\section{Appendix: Existence of the moduli space}\label{appendix}

We will now see that the space  $\mathcal{H}^n_0$ defined in Section~\ref{QPHBs}, of isomorphism classes of stable rank-$2$ quasi-parabolic $SL(2,\C)$-Higgs bundles over $\C \P^1$ at $D$ for which the underlying holomorphic vector bundle is trivial,  is a moduli space and a complex manifold. For this we use the similar proof given by Furuta and Steer for quasi-parabolic bundles in \cite{FS} and adapt it to quasi-parabolic Higgs bundles.

Let $T$ be a  complex manifold. A holomorphic family $(E_t,\phi_t)_{t\in T}$ of  stable rank-$2$  trivial holomorphic quasi-parabolic Higgs bundles over $\C\P^1$ at $D=\{p_1,\ldots,p_n\}$ parametrized by $T$ is
\begin{itemize}
\item a holomorphic family of stable rank-$2$ trivial quasi-parabolic vector bundles $\widetilde{E} \to \C\P^1 \times T$ para\-metrized by $T$ (see \cite{FS})\footnote{i.e a holomorphic rank-$2$ vector bundle $\widetilde{E} \to \C\P^1 \times T$ such that $\widetilde{E}_{\lvert_{ \{ p_i \}\times T}}  \simeq \varphi^* \widetilde{E} \to T$  is a rank-$2$ vector bundle over $T$ (where $\varphi_i:T\to \C\P^1 \times T$ is the map $\varphi_i(t)=(p_i,t)$ for  $i=1,\ldots,n$) and a filtration
$$
\varphi_i^* \widetilde{E} =  \widetilde{E}_{i,1}\supsetneq  \widetilde{E}_{i,2}\supsetneq \{0\}, 
$$ 
for each $i=1,\ldots,n$, where $\widetilde{E}_{i,2} \to T$ is a line subbundle of $\varphi^* \widetilde{E}$.};
\item a holomorphic section $\Phi$ of $SQPar End(\widetilde{E}) \otimes \rho^* K_{\C \P^1} (D)$, where $\rho:\C\P^1 \times T \to \C\P^1$ is the projection map, such that
$$
\Phi\lvert_{\C\P^1\times\{t\}} = \phi_t, \quad \forall t\in T. 
$$
Note that $\Phi$ is an $End(\widetilde{E})$-valued $1$-form on $\C\P^1\setminus D$ as $\rho^* K_{\C\P^1}(D)\subset K_{\C\P^1\times T} (\rho^* D)$.
\end{itemize}
Given such a family we have a set theoretical map 
\begin{align*}
f: T & \to \mathcal{H}^n_0 \\
t & \mapsto [(E_t,\phi_t)]
\end{align*}
classifying isomorphism classes of parametrized bundles. We will start by constructing  a collection of holomorphic families  of stable trivial holomorphic quasi-parabolic Higgs bundles whose images under the map $f$ cover  the 
set ${\mathcal H}_0^n$.

\begin{proposition}\label{prop:appprop1} There exists a finite collection of holomorphic families of stable rank-$2$ trivial holomorphic quasi-parabolic Higgs bundles over $\C\P^1$ at $D$ parametrized by open sets of the form $\C^{n-3}\times (\C^*)^{n-3}$ whose images under the map $f$ cover 
the set $\mathcal{H}^n_0$. 
\end{proposition}

\begin{proof}
Let $\Sigma$ be $\C \P^1$ and consider a set of $n$ marked points $\{x_1,\ldots,x_n\}\neq \varnothing$ in $\Sigma$. Let $E$ be a stable rank-$2$ quasi-parabolic bundle over $\Sigma$ at $D$ for which the underlying holomorphic vector bundle is 
trivial (then
$
E\simeq \mathcal{O}(0)\oplus \mathcal{O}(0)$).

For a fixed rank-$2$ holomorphic trivial vector bundle $E$,  quasi-parabolic structures  as in \eqref{eq:rk2qpb} are parametrized by a product of flag manifolds 
$$
\prod_{i=1}^n (Iso(\C^2, E_{x_i,1})/B) \simeq \left(\C \P^1\right)^n,
$$
 where $Iso(\C^2, E_{x_1,1})$ is the set of linear isomorphisms from $\C^2$ to $E_{x_1,1}$, and $B$ is the parabolic subgroup of $SL(2,\C)$ of upper triangular matrices (which leave the standard full flag in $\C^2$ invariant). 
 
Moreover, the automorphism group of $E$ is $SL(2,\C)$ and, by Remark~\ref{rem:stable}, a point $(X_1,\ldots, X_n)\in \left(\C \P^1\right)^n$ represents a stable quasi-parabolic bundle if and only if 
$$\lvert \{ X_1,\ldots, X_n\} \rvert \geq 2.$$
 

Considering disjoint index sets $A,B\subset \{1,\ldots,n\}$ such that $A\cup B= \{1,\ldots,n\}$, we can cover $\left(\C \P^1\right)^n$ by a finite number of coordinate charts
 $$
 U_{A,B} := \prod_{i\in A} \{[1:w_i]:\,\, w_i \in \C\} \times \prod_{i\in B}  \{[z_i:1]:\,\, z_i \in \C\} \simeq \C^{n}.
 $$
 
Now the possible Higgs fields on a quasi-parabolic bundle represented by a point in $U_{A,B}$ are determined by the collections of their residue matrices. These are traceless nilpotent  matrices $N_1,\ldots, N_n$ such that $N_i(E_{x_i,1}) = E_{x_i,2}$.  In particular, if the quasi-parabolic Higgs bundle is stable, $E_{x_i,2}$ is the kernel of $N_i$ (since, by stability, we have $N_i\neq 0$). Each of these matrices is of the form 
 \begin{equation}\label{eq_resmat}
 N_i=  \lambda_i \left(\begin{array}{lc} - w_i & 1 \\ -w_i^2 & w_i \end{array} \right), \quad \text{if  $i\in A$}  \quad \text{and} \quad  N_i=  \lambda_i \left(\begin{array}{lc} - z_i & z_i^2 \\ -1 & z_i \end{array} \right), \quad \text{if  $i\in B$},
 \end{equation}
with $\lambda_i\in \C^*$, when the quasi-parabolic bundle is represented by a point 
$$\prod_{i\in A} [1:w_i] \times \prod_{i\in B} [z_i:1] \in U_{A,B}.$$

Assuming, without loss of generality, that $X_1\neq X_2$, there is a unique element of $SL(2,\C)/\{ \pm Id\}$ that takes 
 $X_1$ to $[1:0]$,  $X_2$ to $[0:1]$ and 
$N_1$ to the matrix
$$ N_1= \left(\begin{array}{lc} 0 & 1\\ 0 & 0 \end{array} \right).$$
Note that this element of $SL(2,\C)/\{ \pm Id\}$ takes $N_2$ to a matrix of the form
$$ N_2= \lambda_2 \left(\begin{array}{cl} 0 & 0\\ - 1 & 0 \end{array} \right), \quad \text{with $\lambda_2\in \C^*$}.$$

Consequently, a point 
$$\left(([1:0],[0:1],X_3,\ldots,X_n), 1,\lambda_2,\ldots,\lambda_n\right)\in (\C \P^1)^n\times (\C^*)^n,$$ 
with $X_i=[1:w_i]$ for $i \in A$, $X_i=[z_i:1]$ for $i\in B$  and $\lambda_i\in \C^*$, represents an isomorphism class of a stable quasi-parabolic Higgs bundle whose underlying quasi-parabolic bundle is represented by  a point in $U_{A,B}$ with $1\in A$ and $2\in B$ (and $w_1=z_2=0$).

Since the sum of the residues must be zero, the $z_i,w_i,\lambda_i$ must satisfy the following compatibility conditions
 
 

\begin{align}\label{eq:sumzero}
& 1+ \sum_{i\in A\setminus\{1\}} \lambda_i + \sum_{i\in B\setminus\{2\}} z_i^2 \lambda_i  = 0 \\ \label{eq:sumzero2}
& \sum_{i\in A\setminus\{1\}} w_i \lambda_i + \sum_{i\in B\setminus\{2\}} z_i \lambda_i  = 0 \\ \label{eq:sumzero3}
& \sum_{i\in A\setminus\{1\}} w_i^2  \lambda_i +  \sum_{i\in B}  \lambda_i   = 0.
\end{align}

If $A\neq \{1\}$ then, by \eqref{eq:sumzero} and \eqref{eq:sumzero2}, we can take one $i_0 \in A\setminus \{1\}$ and determine
\begin{equation}\label{eq:i01}
\lambda_{i_0}=- \left(1 +  \sum_{i\in A\setminus\{1,i_0\}} \lambda_i + \sum_{i\in B\setminus\{2\}} z_i^2 \lambda_i\right) \neq 0
\end{equation} 
and 
\begin{equation}\label{eq:i02}
w_{i_0} = \frac{1}{1 +  \sum_{i\in A\setminus\{1,i_0\}} \lambda_i + \sum_{i\in B\setminus\{2\}} z_i^2 \lambda_i}  \left( \sum_{i\in A\setminus\{1,i_0\}} w_i \lambda_i + \sum_{i\in B\setminus\{2\}} z_i \lambda_i  \right).
\end{equation}
Moreover, $\lambda_2$ can be determined by $\eqref{eq:sumzero3}$. 

If $A= \{1\}$, then by  \eqref{eq:sumzero}, there exists $i_0 \in B\setminus \{2\}$ such that $z_{i_0}\neq 0$ and so we can determine $\lambda_{i_0}$ and $z_{i_0}$ from \eqref{eq:sumzero} and \eqref{eq:sumzero2}:
\begin{align}\label{eq:i03}
z_{i_0} & = \frac{1 + \sum_{i\in B\setminus\{2,i_0\}} z_i^2 \lambda_i}{ \sum_{i\in B\setminus\{2,i_0\}} z_i \lambda_i} \\
\label{eq:i04}
\lambda_{i_0} & = -\frac{\left( \sum_{i\in B\setminus\{2,i_0\}} z_i \lambda_i\right)^2}{1 + \sum_{i\in B\setminus\{2,i_0\}} z_i^2 \lambda_i}.
\end{align}
Again $\lambda_2$ can be determined by $\eqref{eq:sumzero3}$. 

Consequently, the quotient by $SL(2,\C)/\{ \pm Id\}$  of the set $H_{A,B}$ of trivial quasi-parabolic Higgs bundles with quasi-parabolic structures given by $ (X_1,\ldots,X_n) \in U_{A,B}$, such that
\begin{align} \label{eq:conditions1}
& \bullet \text{$X_1\neq X_2$}, \\ \label{eq:conditions2}
& \bullet \text{if $X_i=X_1$ then $i\in A$ and, if $X_i=X_2$ then $i\in B$ and} \\ \label{eq:conditions3}
& \bullet \text{Higgs field defined by residues of the form \eqref{eq_resmat} determined by} \\ \nonumber
& \text{$n$ nonzero complex numbers  $\lambda_1,\ldots,\lambda_n\in \C^*$ satisfying   \eqref{eq:sumzero},  \eqref{eq:sumzero2} and  \eqref{eq:sumzero3}},
\end{align} 
can be identified with
\begin{align}\label{eq:local}
 \{ ([1:0],[0:1], & X_3, \ldots, X_n,1,\lambda_2,\ldots, \lambda_n )   \in  U_{A,B} \times (\C^*)^n  : \\ \nonumber &  \text{$X_i=[1: w_i]$ if $i\in A \setminus \{1\}$, $X_i=[z_i:1]$ if $i\in B\setminus\{2\}$,}  \\ \nonumber & \text{with $z_i, w_i \in \C$ and  $\lambda_i \in \C^*$ satisfying   \eqref{eq:sumzero}, \eqref{eq:sumzero2} and  \eqref{eq:sumzero3},} \\ \nonumber
 \nonumber & \text{for $w_1=z_2=0$ and $\lambda_1=1$} \} \simeq \C^{n-3}\times (\C^*)^{n-3}.
\end{align}

\begin{rem}\label{rem:covering}
Considering injective functions $\sigma:\{1,2\} \to \{1,\ldots,n\}$ with $\sigma(1)<\sigma(2)$ and the sets 
$$U^\sigma_{A,B} = \{ (X_1,\ldots,X_n)\in (\C \P^1)^n: X_{\sigma(1)} \neq X_{\sigma(2)} \}$$
as well as the corresponding sets of stable quasi-parabolic Higgs bundles $H_{A,B}^\sigma$, we conclude that the set $H_0$ of stable rank-$2$ quasi-parabolic bundles over $\Sigma$ at $D$ for which the underlying holomorphic vector bundle is 
trivial is covered by 
$$
\bigcup_{\sigma} \,\, \bigcup_{A,B} \,\, H^\sigma_{A,B}. 
$$
\end{rem}

As we have seen before, the set of isomorphism classes of elements of each set  $H_{A,B}^\sigma$ can be identified with $\C^{n-3}\times (\C^*)^{n-3}$ and so,
by Remark~\ref{rem:covering}, we conclude that the space $\mathcal{H}^n_0$  of isomorphism classes of rank-$2$ quasi-parabolic Higgs bundles over $\Sigma$ at $D$ for which the underlying holomorphic vector bundle is
trivial, can be covered  by a finite number of copies of  $\C^{n-3}\times (\C^*)^{n-3}$. Moreover,  the patching maps among these sets are holomorphic. 
\end{proof}

\begin{rem}
 It is easy to check that $\mathcal{H}^n_0$ with the quotient topology is Hausdorff and so Proposition~\ref{prop:appprop1} shows that $\mathcal{H}^n_0$ has the structure of a complex manifold of (real) dimension $4(n-3)$ covered by finitely many open sets of the form  $\C^{n-3}\times (\C^*)^{n-3}$.
\end{rem}

Let $T$ be any complex manifold parametrizing a holomorphic family of rank-$2$ trivial holomorphic quasi-parabolic Higgs bundles over $\C \P^1$ at $D$ and consider the corresponding map $f:T \to  \mathcal{H}^n_0$ classifying  isomorphism classes of parametrized bundles. We will see that $f$ is holomorphic and so $\mathcal{H}^n_0$ is a moduli space.


\begin{proposition} Given a complex manifold $T$ parametrizing a holomorphic family $(E_t,\phi_t)_{t\in T}$ of stable rank-$2$ trivial holomorphic quasi-parabolic Higgs bundles over $\C \P^1$ at $D$, the classifying map $f$  is holomorphic.
\end{proposition}
\begin{proof} 

Since $\widetilde{E}$ is a rank-$2$ vector bundle over $\C\P^1\times T$, we know that $\mathcal{O}(\widetilde{E})$ is a rank-$2$ locally free sheaf over $\C\P^1\times T$.
Let $\pi:\C\P^1 \times T \to T$ be the projection map and let
$$
(\C\P^1\times T)_t:=\pi^{-1}(t) = \C \P^1\times \{t\} \quad \text{and}\quad \mathcal{O}(\widetilde{E})_t:=\mathcal{O}(\widetilde{E})\lvert_{(\C \P^1\times T)_t}.
$$
Then
$$
\dim_{\C} H^0 ((\C\P^1\times T)_t,\mathcal{O}(\widetilde{E})_t) = \dim_{\C} H^0 (\C\P^1\times \{t\}, \mathcal{O} (\widetilde{E})\lvert_{\C \P^1\times\{t\}})= 2
$$
(as $\mathcal{O} (\widetilde{E})\lvert_{\C \P^1\times\{t\}}=\mathcal{O}(\widetilde{E}\lvert_{\C\P^1\times \{t\}})$ and $\widetilde{E}_{\lvert_{\C \P^1\times\{t\}}}\simeq \C\P^1\times \C^2$).

Consequently,  the sheaf $R^0\pi_*(\widetilde{E})$  over $T$ defined by
$$
R^0\pi_*(\widetilde{E})(U) = H^0(\pi^{-1}(U),\widetilde{E}),
$$
for open Zariski open subsets $U$ of $T$, is a rank-$2$  locally free sheaf and so there is an open cover $\{U_i\}_i$ of $T$ such that
$$
R^0\pi_*(\widetilde{E})\lvert_{U_i} \simeq \mathcal{O}_T(U_i) \oplus  \mathcal{O}_T(U_i) \simeq U_i\times \C^2.
$$ 
This gives a trivialization of $R^0\pi_*(\widetilde{E})$ on $U_i$
$$
\varphi_i: R^0\pi_*(\widetilde{E})\lvert_{U_i} \to U_i\times \C^2,
$$ 
which, in turn, gives a trivialization of  $\widetilde{E}$ on the product $\C\P^1\times U_i$,
$$
\psi_i: \widetilde{E}\lvert_{\C\P^1\times U_i} \to (\C\P^1\times U_i) \times \C^2.
$$
Indeed, 
$$
R^0\pi_*(\widetilde{E})(U_i) = H^0(\pi^{-1}(U_i),\widetilde{E}) = H^0(\C\P^1 \times U_i, \widetilde{E})
$$
and, since $\varphi_i$ is a trivialization, we can define 
\begin{align*}
\varphi_i^{-1}: U_i\times \C^2 & \to R^0\pi_*(\widetilde{E}_i)\lvert_{U_i}  \\
(t,(z,w)) &\mapsto  \varphi_i^{-1}(t,(z,w)): \C \P^1 \times U_i \to \widetilde{E}\lvert_{\C\P^1\times U_i},
\end{align*}
as well as the isomorphism
\begin{align*}
\psi_i^{-1}:(\C\P^1 \times U_i) \times \C^2 & \to \widetilde{E}\lvert_{\C\P^1\times U_i} \\
((p,t),(z,w)) &\mapsto  \varphi_i^{-1}(t,(z,w))(p,t).
\end{align*}
Hence, there exists an open subset $U$ of $T$ such that $\widetilde{E}\lvert_{\C\P^1\times U} \simeq (\C\P^1\times U) \times \C^2$
and this set $U$ parametrizes holomorphically trivial stable quasi-parabolic Higgs bundles over $\C \P^1$ at $D$. In particular, we have a map
\begin{align*}
\widetilde{f}: U & \to (\C \P^1)^n \times (\C^*)^n\\
\widetilde{f} & \mapsto (X_1(t),\ldots, X_n(t), \lambda_1(t),\ldots, \lambda_n(t))
\end{align*}
with 
$$
\widetilde{f}(U) \subseteq \{(X_1,\ldots, X_n, \lambda_1,\ldots, \lambda_n)\in (\C \P^1)^n \times (\C^*)^n : \,\, \lvert \{ X_1,\ldots, X_n\} \rvert \geq 2\}.
$$

Assuming, without loss of generality, that for each $t\in U$ we have $X_1(t)\neq X_2(t)$ and $(X_1(t),\ldots, X_n(t))\in U_{A,B}$ for some $A,B$ satisfying \eqref{eq:conditions1}, \eqref{eq:conditions2} and \eqref{eq:conditions3},  it is enough to show that the quotient map
from the set $H_{A,B}$ 
which sends $(X_1,\ldots,X_n)$ to 
$$(gX_3,\ldots,\ldots,\widehat{gX_{i_0}},\ldots, gX_n)$$ and $(N_1,\ldots, N_n)$ to $(g N_3 g^{-1},\ldots,\widehat{g N_{i_0} g^{-1}},\ldots, g N_n g^{-1})$, with  $gX_1=[1:0]$, $g X_2= [0:1]$, 
$$g N_1 g^{-1} = \left( \begin{array}{cc} 0 & 1 \\ 0 & 0 \end{array} \right) \quad \text{and}\quad g N_2 g^{-1} = \widetilde{\lambda}_2 \left( \begin{array}{rc}  0 & 0 \\ -1 & 0 \end{array} \right)
$$  
with $\widetilde{\lambda}_2$ determined by $\eqref{eq:sumzero3}$
(and consequently sends $(\lambda_1,\ldots,\lambda_n)$ to 
$$(g\cdot \lambda_3,\ldots,\widehat{g \cdot  \lambda_{i_0}},\ldots,g \cdot\lambda_n)$$ 
with $g \cdot\lambda_1=1$ and $g \cdot\lambda_2= \widetilde{\lambda}_2 = (1-w_1z_2)^2 \lambda_1 \lambda_2$)  is holomorphic.

Here 
$$
g\cdot \lambda_i=\frac{ (1-z_2w_i)^2}{\lambda_1(1-z_2w_1)^2} \, \lambda_i = \left(g \left(\begin{array}{c} 1 \\ w_i \end{array}\right)_1\right)^2 \lambda_i, \quad \text{if $i \in A$, where $w_i$ is such that $X_i=[1:w_i]$},
$$
$$
g\cdot \lambda_i= \lambda_1(1-w_1z_i)^2 \,\lambda_i = \left(g \left(\begin{array}{c} z_i \\ 1 \end{array}\right)_2\right)^2 \lambda_i, \quad \text{if $i \in B$, where $z_i$ is such that $X_i=[z_i:1]$},
$$
and $i_0$ is the index used in \eqref{eq:i01},  \eqref{eq:i02},  \eqref{eq:i03} and  \eqref{eq:i04}. The corresponding values of $X_{i_0}$ and $\lambda_{i_0}$ are determined by the equations \eqref{eq:sumzero}, \eqref{eq:sumzero2} and  \eqref{eq:sumzero3}, a consequence of the fact that the sum of the residues is zero.

The quotient map is the composition $h_3 \circ h_2\circ h_1$ of maps
\begin{align*}
(X_1,\ldots, X_n, \lambda_1,\ldots, \lambda_n ) & \stackrel{h_1}{\mapsto}  \left( [1:0],[0:1],  gX_3, \ldots, gX_n, 1, (1-w_1z_2)^2 \lambda_1 \lambda_2, g\cdot \lambda_3, \ldots, g \cdot \lambda_n \right), 
\end{align*}
\begin{align*}
(X_1,\ldots, X_n, \lambda_1,\ldots, \lambda_n ) & \stackrel{h_2}{\mapsto}  (X_3, \ldots, X_n,\lambda_2,\ldots, \lambda_n )  
\end{align*}
and
\begin{align*}
(X_3,\ldots, X_n, \lambda_2,\ldots, \lambda_n ) & \stackrel{h_3}{\mapsto}  (a_3, \ldots, \hat{a}_{i_0},\ldots, a_n,\lambda_3,\ldots, \hat{\lambda}_{i_0},\ldots  \lambda_n )  \in \C^{n-3}\times (\C^*)^{n-3}
\end{align*}
where, for $i\in A$, we have $X_i=[1:w_i]$ and $a_i=w_i$ and, for $i\in B$, we have $X_i=[z_1:1]$ and $a_i=z_i$. Moreover, $f_{\lvert_{U}} =  h_3 \circ h_2\circ h_1\circ \widetilde{f}$.

Note that the map $h_1$ is holomorphic since the map $g$ is holomorphic and depends holomorphicaly on $X_1$, $X_2$ and $\lambda_1$:

$$
g= \pm  \left( \begin{array}{cc} \displaystyle{\frac{1}{  \lambda_1^{1/2} (1-w_1z_2)} }&  \displaystyle{-\frac{z_2 }{\lambda_1^{1/2} (1-w_1z_2)} }\\ \\ - \lambda_1^{1/2} w_1&   \lambda_1^{1/2} \end{array} \right) \in SL(2,\C)/\{ \pm Id\}. 
$$
\end{proof}

\end{document}